\documentclass[a4paper,12pt]{amsart}

\usepackage{amssymb}
\usepackage{amsmath}
\usepackage{amsthm}
\usepackage{mathtools}
\usepackage{enumitem}
\usepackage{graphicx}
\usepackage{tikz}
\usepackage{mathrsfs}
\usepackage[mathscr]{euscript}
\usepackage{tkz-graph}

\usepackage{caption}
\usepackage[skip=1ex, belowskip=2ex]{subcaption}
\usepackage[export]{adjustbox}

\usetikzlibrary{arrows}
\usetikzlibrary{tikzmark}
\usetikzlibrary{cd}
\usetikzlibrary{calc}
\graphicspath{ {.} }

\usepackage[english]{babel}

\newcommand{\innerthmname}{}

\theoremstyle{definition}

\usetikzlibrary{positioning}
\newtheorem{theorem}[equation]{Theorem}
\newtheorem{lemma}[equation]{Lemma}

\newtheorem{corollary}[equation]{Corollary}

\theoremstyle{definition}
\newtheorem{definition}[equation]{Definition}

\newcounter{condition}

\theoremstyle{remark}
\newtheorem{example}[equation]{Example}
\newtheorem{remark}[equation]{Remark}

\newtheorem{condition}[condition]{Condition}

\numberwithin{equation}{section}
\numberwithin{condition}{section}
\usepackage{fancyhdr}
\usepackage{amsfonts}
\usepackage{hyperref}
\usepackage{graphicx, xcolor} 
\usepackage{amsthm}

\usepackage[
margin=1in,
marginpar=2cm,
includefoot,
footskip=30pt,
]{geometry}

\usepackage{scalerel,stackengine}
\stackMath
\newcommand\reallywidehat[1]{%
	\savestack{\tmpbox}{\stretchto{%
			\scaleto{%
				\scalerel*[\widthof{\ensuremath{#1}}]{\kern-.6pt\bigwedge\kern-.6pt}%
				{\rule[-\textheight/2]{1ex}{\textheight}}
			}{\textheight}%
		}{0.5ex}}%
	\stackon[1pt]{#1}{\tmpbox}%
}


\title[Morita Equivalence of Subrings]{Morita Equivalence of Subrings with Applications to Inverse Semigroup Algebras}
\author[A. Zhang]{Allen Zhang}
\email{allenusca@gmail.com}
\begin{document}
	
	\begin{abstract}
		We develop a technique to show the Morita equivalence of certain subrings of a ring with local units. We then apply this technique to develop conditions that are sufficient to show the Morita equivalence of subalgebras induced by partial subactions on generalized Boolean algebras and, subsequently, strongly $E^{\ast}$-unitary inverse subsemigroups. As an application, we prove that the Leavitt path algebra of a graph is Morita equivalent to the Leavitt path algebra of certain subgraphs and use this to calculate the Morita equivalence class of some Leavitt path algebras. Finally, as the main application, we prove a desingularization result for labelled Leavitt path algebras.
	\end{abstract}
	\maketitle
	
	\section{Introduction}	
	Morita equivalence was first defined for rings by Morita in his seminal paper \cite{Morita1958DualityFM}. Since then, the theory of Morita equivalence has been extended to a variety of other objects. Rieffel defined a notion of strong Morita equivalence for $C^{\ast}$-algebras in \cite{RIEFFEL1974176}. In the study of $C^{\ast}$-algebras, it was found that groupoids and inverse semigroups were deeply intertwined in the theory of $C^{\ast}$-algebras, so theories of Morita equivalence were developed for inverse semigroups and groupoids and subsequently related back to the strong Morita equivalence of $C^{\ast}$-algebras \cite{ Lawson1996EnlargementsOR, Muhly87, SteinbergSemiMorita, Talwar_1995}.
	
	A different notion of Morita equivalence was defined for rings with local units by Abrams in \cite{Abrams01011983} and then further elaborated on by Ánh and Márki in \cite{AnhMoritaEquivalenceLocal}. This notion of Morita equivalence was then applied to Leavitt path algebras and proved to have important applications in the field of Leavitt path algebras. For a survey of the history of Leavitt path algebras and some applications of Morita equivalence in it, see the survey by Abrams \cite{AbramsSurvey}. Clark and Sims showed in \cite{CLARK20152062} that the theory of Morita equivalence of groupoids developed for $C^{\ast}$-algebras was also useful for the algebraic analog, known as Steinberg algebras.
	
	In this paper, we develop conditions that are sufficient to prove that certain subrings $S \subseteq R$ for rings with local units are Morita equivalent. Our result is a broad generalization of a technique used to prove desingularization results in \cite{abrams2008leavitt, Firrisa2020MoritaEO}. As a result of our larger generality, we find parallels to the notion of a full hereditary $C^{\ast}$-subalgebra $B \subseteq A$, which was proven by Rieffel in \cite{rieffel1982morita} to be sufficient to conclude that $B$ and $A$ are Morita equivalent as $C^{\ast}$-algebras.
	
	A key contribution of our paper will be to develop our results for objects which induce our $R$-algebras, which decreases the	computational intensity of the proofs in previous papers. As our main result, we will prove a desingularization result for labelled Leavitt path algebras, which were defined in \cite{Boava2021LeavittPA}. We will accomplish this by realizing our subalgebras from subobjects in the following order, and then proving conditions sufficient for Morita equivalence for each of these in succession.
	
			\tikzset{
		basic/.style={
			draw,
			rectangle,
			align=center,
			minimum width=8em,
			minimum height=3em,
			text centered
		},
		arrsty/.style={
			draw=black,
			-latex
		}
	}
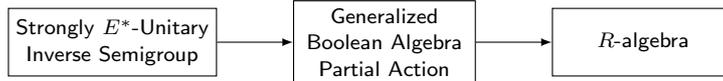
	\begin{center}
		\begin{tikzpicture}[every node/.style={font=\sffamily}]
			\tikzset{every node}=[font=\tiny]
			
			\node[basic] (b) {Strongly $E^{\ast}$-Unitary \\ Inverse Semigroup};
			\node[basic, right=of b] (c) {Generalized \\ Boolean Algebra \\ Partial Action };
			\node[basic, right=of c] (d) {$R$-algebra};
			
			\draw[arrsty]   (b)    --  (c);
			\draw[arrsty]   (c)    --  (d);
			
		\end{tikzpicture}
	\end{center}

	The theory for relating the above objects was primarily developed by this paper's author in \cite{zhang2025partialactionsgeneralizedboolean} and this paper borrows much of the same notation.

	This idea of proving Morita equivalence through successive realizations is not new, as there are well-developed theories of Morita equivalence of various objects that induce each other \cite{CLARK20152062, Muhly87, Renault87, SteinbergSemiMorita}. However, and our hope is that the conditions in this paper are easier to check in practice and provide further insight into how these objects relate to each other on the level of combinatorial objects.
	
	The outline of the paper is as follows. In Section~\ref{moritaequivsect}, we extend a technique from \cite{abrams2008leavitt, Firrisa2020MoritaEO} for proving Morita equivalence for subrings with $\sigma$-units to subrings with arbitrary local units and derive some tractable conditions where the technique is guaranteed to work. In Section~\ref{booleanaction}, we extend these conditions to conditions on partial subactions on a generalized Boolean algebra. In Section~\ref{inversegroupaction}, we further extend these conditions to conditions on strongly $E^{\ast}$-unitary inverse subsemigroups. In Section~\ref{enlargement}, we will show how the notion of enlargement defined by Lawson in \cite{Lawson1996EnlargementsOR} is related to our conditions. Finally, in Section~\ref{Leavittapp} and \ref{AppLabelled}, we apply our results to prove Morita equivalence for various combinatorial $R$-algebras, including a desingularization result for labelled Leavitt path algebras.
	
	\section{Preliminaries}
	For ease of reference, in this section, we recreate the definitions and constructions presented in \cite{zhang2025partialactionsgeneralizedboolean}. For a more in-depth discussion, refer to the aforementioned paper.
	
	\begin{definition}[{\cite[Section~2.1]{zhang2025partialactionsgeneralizedboolean}}]
		
		$\mathcal B$ is a \textit{generalized Boolean algebra} if it is a lattice that is distributed, relatively complemented, and has a minimal element $0$. Morphisms of a generalized Boolean algebra are maps of lattices that preserve the operations and bottom element. A subset $\mathcal I \subseteq \mathcal B$ is called an ideal if for all $a, b \in \mathcal I$ we have that $a \vee b \in \mathcal I$ and for all $a \in \mathcal I$ and $b \in \mathcal B$ we have that $a \wedge b \in \mathcal I$. A \textit{cover} $S \subseteq \mathcal B$ of $\mathcal B$ is a set such that any element $x \in B$ is less than the finite union of some $\bigvee_{i=1}^n s_i$ for $s_i \in C$. A \textit{generator} $S \subseteq \mathcal B$ of $\mathcal B$ is a set where every element in $\mathcal B$ can be formed from some finite sequence of operations of elements from $S$.
		
		A partial action of a group $G$ on a generalized Boolean algebra $\mathcal B$ is a pair $\Phi = (\{\mathcal I_t\}_{t \in G}, \{\phi_t\}_{t \in G})$ where $\mathcal I_t \subseteq \mathcal B$ are ideals and $\phi_t: \mathcal I_{t^{-1}} \rightarrow \mathcal I_t$ are generalized Boolean algebra isomorphisms such that the following properties are satisfied: 
		\begin{enumerate}
			\item $\mathcal I_e = \mathcal B$ and $\phi_e$ is the identity on $\mathcal B$
			\item $\phi_s(\mathcal I_{s^{-1}} \cap \mathcal I_t) = \mathcal I_s \cap \mathcal I_{st}$ for all $s, t \in G$
			\item $\phi_s(\phi_t(x)) = \phi_{st}(x)$ for $x \in \mathcal I_{s^{-1}} \cap \mathcal I_{(st)^{-1}}$ for all $s, t \in G$
		\end{enumerate}
		We often also refer to $(\mathcal B, \Phi)$ as a partial action on a generalized Boolean algebra to keep the set $\mathcal B$ in the notation. Morphisms of partial actions of a group $G$ from $(\mathcal B_1, \Phi_1) \rightarrow (\mathcal B_2, \Phi_2)$ are formed by generalized Boolean algebra morphisms $f: \mathcal B_1 \rightarrow \mathcal B_2$ that satisfy the following properties.
		\begin{enumerate}
			\item $f(\mathcal I_{1, t}) \subseteq \mathcal I_{2, t}$ for all $t \in G$
			\item $f$ induces the following commutative diagram for all $t \in G$
			\begin{center}
				\begin{tikzcd}
					\mathcal I_{1, t^{-1}} \arrow[r, "f"] \arrow[d, "\phi_{1, t}"] & \mathcal I_{2, t^{-1}} \arrow[d, "\phi_{2, t}"] \\
					\mathcal I_{1, t} \arrow[r, "f"] & \mathcal I_{2, t} \\
				\end{tikzcd}
			\end{center}
			
		\end{enumerate}
		If $f$ is injective, we say that $(\mathcal B_1, \Phi_1) \subseteq (\mathcal B_2, \Phi_2)$ is a partial subaction. 
		
	\end{definition}

	The specifics of the next definition are not particularly important, so we merely define the notation and refer to \cite[Lemma~2.24]{zhang2025partialactionsgeneralizedboolean} for specific calculations involving elements.

	\begin{definition}[{\cite[Section~2.2]{zhang2025partialactionsgeneralizedboolean}}]
		To a generalized Boolean algebra $\mathcal B$ and a commutative unital ring $R$, we associate an $R$-algebra $\mathrm{Lc}(R, \mathcal B)$. To a partial action $\Phi = (\{\mathcal I_t\}_{t \in G}, \{\phi_g\}_{t \in G})$ of a group $G$ on a generalized Boolean algebra $\mathcal B$, we associate a dual action $(\{\mathrm{Lc}(\mathcal I_g, R)\}_{g \in G}, \{\tilde{\phi}_g\}_{g \in G})$, which induces a partial skew group ring denoted as $\mathrm{Lc}(\mathcal B, R) \rtimes_{\Phi} G$. This $R$-algebra is $R$-spanned by elements of the form $\{U \delta_g\}_{g \in G, U \in \mathcal I_g}$ and calculations involving these elements are described in \cite[Lemma~2.23]{zhang2025partialactionsgeneralizedboolean}.
		
		Morphisms from $(\mathcal B_1, \Phi_1) \rightarrow (\mathcal B_2, \Phi_2)$ induce morphisms $\mathrm{Lc}(R, \mathcal B_1) \rtimes_{\Phi_1} G \rightarrow \mathrm{Lc}(R, \mathcal B_2) \rtimes_{\Phi_2} G$. Furthermore, if $(\mathcal B_1, \Phi_1) \subseteq (\mathcal B_2, \Phi_2)$, we get that $\mathrm{Lc}(R, \mathcal B_1) \rtimes_{\Phi_1} G \subseteq \mathrm{Lc}(R, \mathcal B_2) \rtimes_{\Phi_2} G$ as a homogenous $G$-graded $R$-algebra.
	\end{definition}

	\begin{lemma} \label{lemma:ident}
	Let $(\mathcal B_1, \Phi_1) \subseteq (\mathcal B_2, \Phi_2)$ be a partial subaction of a group $G$ on generalized Boolean algebras. Then $\mathcal I_{1, g} \subseteq \mathcal I_{2, g}$ for all $g \in G$. Furthermore, the elements $\mathrm{Lc}(R, \mathcal B_1) \rtimes_{\Phi_1} G \subseteq \mathrm{Lc}(R, \mathcal B_2) \rtimes_{\Phi_2} G$ are exactly the $R$-span of $\{U \delta_g\}_{g \in G, U \in \mathcal I_{1, g}}$ where we take the identification $\mathcal I_{1, g} \subseteq \mathcal I_{2, g}$.
	\end{lemma}
	
	\begin{proof}
		$\mathcal I_{1, g} \subseteq \mathcal I_{2, g}$ is obvious from the definition of $(\mathcal B_1, \Phi_1) \subseteq (\mathcal B_2, \Phi_2)$. The identification of the subalgebra follows from \cite[Theorem~2.31]{zhang2025partialactionsgeneralizedboolean}.
	\end{proof}
	
	\begin{definition}[{\cite[Section~3]{zhang2025partialactionsgeneralizedboolean}}]
		Let $P$ be a (meet) semilattice with $0$. For $x \in P$, a \textit{cover} of $x$ is a collection $\{x_1, \ldots, x_n\} \subseteq P$ such $x_i \leq x$ and for any $y \in P$ such that $y \leq x$ we have that there is some $i \in [1, n]$ where $y \wedge x_i \neq 0$. A \textit{filter} $\xi \subseteq P$ is any subset such that for any $x, y \in \xi$ we have that $x \wedge y \in \xi$ and for any $x \in \xi$ and $x \leq y \in P$ we have that $y \in \xi$. A \text{tight filter} is any filter $\xi$ such that for any element $x \in \xi$ and any cover $\{x_1, \ldots, x_n\}$ of $x$, we have that there is some $i$ where $x_i \in \xi$.

		The set of all tight filters $T(P)$ has a topology generated by the compact open sets $V^P_x = \{\xi \in T(P) \colon x \in \xi\}$ for any $x \in P$. When the semilattice $P$ is clear, we simply denote the set as $V_x$. We denote the compact open sets of $T(P)$ as $\mathcal T_c(P)$.

		For any subsemilattice $P_1 \subseteq P_2$, we say that $P_1 \subseteq P_2$ preserves finite covers if for any $x \in P_1$ and $\{x_1, \ldots, x_n\}$ a finite cover of $x$ in $P_1$, we have that $\{x_1, \ldots, x_n\}$ is also a finite cover of $x$ in $P_2$. If $P_1 \subseteq P_2$ preserves finite covers, then there is a natural inclusion $\mathcal T_c(P_1) \subseteq \mathcal T_c(P_2)$ given by $V^{P_1}_x \mapsto V^{P_2}_x$. If $\mathcal T_c(P_1) \subseteq \mathcal T_c(P_2)$ is also an ideal, we say that $P_1 \subseteq P_2$ is \textit{tight}.
	\end{definition}

	\begin{definition}[{\cite[Section~4.1]{zhang2025partialactionsgeneralizedboolean}}] \label{definition:stronglyunit}
		
		Let $S$ be an inverse semigroup with $0$ with idempotents $E$, and let $S^{\times} = S \setminus \{0\}$ and $E^\times= E \setminus \{0\}$. $E$ has a natural order defined by $a \leq b \Leftrightarrow ab = a$ and forms a (meet) semilattice with $a \wedge b = ab$.
		
		$S$ is called \textit{strongly $E^{\ast}$-unitary} if there exists a group $G$ and a map $\varphi: S^{\times} \rightarrow G$ where $\varphi^{-1}(1_G) = E^{\times}$ and such that $\varphi(ab) = \varphi(a)\varphi(b)$ for all $a, b \in S^{\times}$ with $ab \neq 0$. We sometimes denote the strongly $E^{\ast}$-unitary inverse semigroup as $(S, \varphi)$ to keep a specific grading in mind.
	\end{definition}

	\begin{definition}[{\cite[Section~4.3]{zhang2025partialactionsgeneralizedboolean}}]
		
		Fix a strongly $E^{\ast}$-unitary inverse semigroup $(S, \varphi)$ with group $G$ and morphism $\varphi: S^{\times} \rightarrow G$. For all $g \in G$, define the sets $E_g = \{x \in E \colon x \leq ss^{\ast} \text{ for some } s \in S \text{ with } \varphi(s) = g\} \subseteq E$ which are all semilattices as well. We define morphisms $\phi_g: E_{g^{-1}} \rightarrow E_g$ and that for $x \in E_{g^{-1}}$ map to $\phi_g(x) = sxs^{\ast}$ where $s \in S$ is any element where $\varphi(s) = g$ and $x \leq s^{\ast}s$. The inverse semigroup $(S, \varphi)$ induces a partial action of $G$ on the generalized Boolean algebra $\mathcal T_c(E)$ with ideals identified with $\mathcal T_c(E_g)$ and actions on ideals $\tilde{\phi}_g: \mathcal T_c(E_{g^{-1}}) \rightarrow \mathcal T_c(E_g)$ generated by $V_x \mapsto V_{\phi_g(x)}$. We denote $\Phi = (\{\mathcal T_c(E_g)\}_{g \in G}, \{\tilde \phi_g\}_{g \in G}\}$ define $L_R(S, \varphi) = \mathrm{Lc}(R, \mathcal T_c(E)) \rtimes_{\Phi} G$. The $R$-algebra $L_R(S, \varphi)$ is generated by elements of the form $\{x \delta_g\}_{g \in G, x \in E_g}$ and calculations involving these elements are described in \cite[Lemma~4.15]{zhang2025partialactionsgeneralizedboolean}. 
		
		As long as $S$ is strongly $E^{\ast}$-unitary, the resulting algebra $L_R(S, \varphi)$ is isomorphic irrespective of the chosen group $G$ and map $\varphi: S^{\times} \rightarrow G$, hence we unambiguously write $L_R(S)$ whenever we don't need to keep a specific grading in mind. For two strongly $E^{\ast}$-unitary inverse semigroups $S_1$ and $S_2$, we say that $S_1 \subseteq_c S_2$ if $S_1 \subseteq S_2$ and $E_1 \subseteq E_2$ preserves finite covers. If $S_1 \subseteq_c S_2$, then for any group $G$ and grading $\varphi: S_2^{\times} \rightarrow G$ on $S_2$ there is an induced grading $\varphi_{\mid S_1^{\times}}$ on $S_1$. Furthermore, the induced partial actions from these gradings satisfy $(\mathcal T_c(E_1), \Phi_1) \subseteq (\mathcal T_c(E_2), \Phi_2)$ and hence $L_R(S_1) \subseteq L_R(S_2)$.
	\end{definition}
	
	\section{Morita Equivalence of Subrings for Rings with Local units}\label{moritaequivsect}
	
	When we refer to Morita equivalence in this paper, we are referring to the notion of Morita equivalence in rings with local units first developed by Abrams in \cite{Abrams01011983} and further explored by Anh and Marki in \cite{AnhMoritaEquivalenceLocal}. We will denote Morita equivalent rings with local units $R$ and $S$ as $R \sim_M S$.
	
	In this section, we present a technique to prove that a subring $S \subseteq R$ of a ring with local units is Morita equivalent. We will later apply this to the partial skew group rings and derive various conditions to make the rings Morita equivalent. Most of this section is a generalization of the ideas presented in \cite{abrams2008leavitt} and \cite{Firrisa2020MoritaEO} to prove a desingularization result for graph and ultragraph algebras respectively.
	
	A significant generalization of this section is removing the countability assumption used previously by being slightly more careful with our directed limits. This lets our results apply to all rings with local units instead of just rings with $\sigma$-units, which will be useful later when we apply our results to objects that aren't necessarily countable.

	\begin{definition}[{\cite[Definition~1]{AnhMoritaEquivalenceLocal}}]
		A ring $R$ (not necessarily unital) is called a \textit{ring with local units} if there exists some set $E \subseteq R$ such that elements of $E$ are idempotents, commute with each other, and $\bigcup_{e \in E} eRe = R$.
		
		For a ring $R$ with local units, the category of (left) $R$-modules is defined to be the modules $M$, with the additional requirement $RM = M$. These are often called unitary modules. Morphisms of $R$-modules are defined in the usual way.
		
		Rings with local units $R$ and $S$ are said to be Morita equivalent if their respective categories $R$-Mod and $S$-Mod of (unitary) modules are equivalent. 
	\end{definition}
	
	All modules we consider in this section will be left modules, so we omit this from our notation.
	
	\begin{example} For a ring $R$ with local units, we present examples of $R$-modules and morphisms that are used throughout this section.
		
		Let $\{M_i\}_{i \in I}$ be a set of $R$-modules over some indexing set $I$. Then \[\bigoplus_{i \in I} M_i = \{(m_i)_{i \in I} \colon m_i \in M_i \text{ and all but finitely many } m_i = 0\}\] is an $R$-module
		
		Let $x \in R$, then \[Rx = \{rx \colon r \in R\}\] is a finitely generated projective $R$-module.
		
		Let $x, y \in R$. Then there is an $R$-module morphism from $Rx \rightarrow Ry$ defined by taking \[rx \mapsto rxy\]
	\end{example}
	\begin{definition}[{\cite[Definition~5.19]{Firrisa2020MoritaEO}}] \label{compatibleset}
		Let $(I,\leq)$ be a partially ordered set and let $\mathcal{C}$ be a category. A \textit{compatible set in $\mathcal{C}$ over $I$}, $\{A_{\alpha},\varphi_{\alpha,\beta},\psi_{\beta,\alpha},I\}$, consists of a set of objects, $\{A_{\alpha}\}_{\alpha\in I}$, and a set of morphisms, 
		$$\{\varphi_{\alpha,\beta}:A_{\alpha}\to A_{\beta},\ \psi_{\beta,\alpha}:A_{\beta}\to A_{\alpha}\}_{\alpha,\beta\in I},$$
		such that:\\
		1) $\varphi_{\alpha,\alpha}=\psi_{\alpha,\alpha}=1_{A_{\alpha}}$ for all $\alpha$,\\
		2) $\varphi_{\beta,\gamma}\circ\varphi_{\alpha,\beta}=\varphi_{\alpha,\gamma}$ and $\psi_{\beta,\alpha}\circ\psi_{\gamma,\beta}=\psi_{\gamma,\alpha}$ for all $\alpha\leq\beta\leq\gamma$,\\
		3) $\psi_{\beta,\alpha}\circ\varphi_{\alpha,\beta}=1_{A_{\alpha}}$ for all $\alpha\leq\beta$,\\
		4) For all $\alpha,\beta,\gamma\in I$ such that $\alpha,\beta\leq\gamma$,
		$$\varphi_{\beta,\gamma}\circ\psi_{\gamma,\beta}\circ\varphi_{\alpha,\gamma}\circ\psi_{\gamma,\alpha}=\varphi_{\alpha,\gamma}\circ\psi_{\gamma,\alpha}\circ\varphi_{\beta,\gamma}\circ\psi_{\gamma,\beta}.$$
		
		Note that the collection $\langle \{A_{\alpha}\}_{\alpha \in I}, \{\varphi_{\alpha, \beta}\}_{\alpha \leq \beta }\rangle$ forms a directed set.
	\end{definition}

	\begin{remark}[{\cite[Proposition~5.24.1]{Firrisa2020MoritaEO}}]\label{endodirect}
		For a compatible set with the category of $R$-Mod for some ring with local units $R$, there is a directed set of $R$-algebras/rings defined by $\langle \mathrm{End}(A_{\alpha}, A_{\alpha})^{\mathrm{op}}, \phi_{\alpha, \beta} \rangle$ where $\phi_{\alpha, \beta}(f) = \varphi_{\alpha, \beta} \circ f \circ \psi_{\beta, \alpha}$.
	\end{remark}
	\begin{definition}[{\cite[Definition~5.24]{Firrisa2020MoritaEO}}]\label{localprojmod} 
		An $R$-module, $M$, is \textit{locally projective} if there exists a compatible set $\{M_{\alpha},\varphi_{\alpha,\beta},\psi_{\beta,\alpha},I\}$ in $R$-mod such that each $M_{\alpha}$ is a finitely generated projective $R$-module and $\langle \{M_{\alpha}\}_{\alpha \in I}, \{\varphi_{\alpha, \beta}\}_{\alpha \leq \beta} \rangle$ satisfies $M\cong\varinjlim\limits_{\alpha\in I}M_{\alpha}$.
	\end{definition}
	
	\begin{definition}[{\cite[Definition~5.22]{Firrisa2020MoritaEO}}] \label{rmodgen} 
		An $R$-module $M$ is a \textit{generator of $R$-mod} if for any $R$-module $N$ there exists a set $I$ and a surjective $R$-module homomorphism
		$$\rho:\bigoplus\limits_{i\in I}M\to N$$
	\end{definition}
	
	\begin{theorem} \label{simpgenerator}
		A module $M$ is a generator of $R$-mod if and only if there exists a set $I$ and a surjective $R$-module homomorphism 		$$\rho:\bigoplus\limits_{i\in I}M\to R$$
	\end{theorem}
	
	\begin{proof}
		Assume that $M$ is an $R$-module, $I$ is an indexing set, and there is a surjective homomorphism  $\rho: \bigoplus_{i \in I} M \rightarrow R$. For any $R$-module $N$, there is a surjective $R$-module homomorphism $\bigoplus_{n \in N} R \rightarrow N$ defined by taking $r_n \mapsto r_n n$. This is surjective because $RN = N$.  There is then a surjective $R$-module homomorphism $\bigoplus_{(i, n) \in I \times N} M \rightarrow \bigoplus_{n \in N} R$ which extends to a surjective morphism to $N$, so we are done.
		
		The other direction is obvious, as $R$ is an $R$-module.
	\end{proof}
	
	Below is the main theorem relating the above definitions and Morita equivalence.
	
	\begin{theorem}[{\cite[Theorem~5.26]{Firrisa2020MoritaEO} and \cite[Theorem~2.5]{AnhMoritaEquivalenceLocal}}] \label{localequiv}
		Let $R$ and $S$ be rings with local units. $R$ and $S$ are Morita equivalent if and only if there exists a locally projective $R$-module $M$ such that $M$ is a generator of $R$-mod and $S \cong \varinjlim\limits_{\alpha \in I} \left(\mathrm{End}_R M_{\alpha}\right)^{op}$.
	\end{theorem}

	\begin{remark}
		We take the opposite ring which differs from the original paper \cite[Theorem~2.5]{AnhMoritaEquivalenceLocal}. This is because, in \cite{AnhMoritaEquivalenceLocal}, $\mathrm{End}_R(M)$ is interpreted as a ring of right actions on $M$. In this paper, we interpret $\mathrm{End}_R(M)$ as actual module morphisms $f: M \rightarrow M$ where multiplication $f \cdot g = f \circ g$ is just composition. The opposite of this ring is exactly the ring in \cite{AnhMoritaEquivalenceLocal}, even though they are denoted the same.
	\end{remark}

	We now generalize a process used in \cite{Firrisa2020MoritaEO, Abrams2017} to arbitrary rings with local units and provide some tractable conditions that are sufficient to conclude that a subring $S \subseteq R$ is Morita equivalent to $R$. For the rest of this section, assume that $R$ is a ring with a set of local units $E$.
	
	\begin{lemma} Let $E$ be a set of local units for a ring $R$ with local units.
		Let $e_1, e_2 \in E$ and define the following join and meet operations.
		
		\[e_1 \vee e_2 = e_1 + e_2 - e_1e_2\]
		\[e_1 \wedge e_2 = e_1e_2\]
		
		If $x, y \in R$ satisfy $e_1x = x$ and $e_2y = y$, then $(e_1 \vee e_2)x = x$ and $(e_1 \vee e_2)y = y$. Furthermore, $E \cup \{e_1 \vee e_2, e_1 \vee e_2\}$ is still a set of local units for $R$. 
	\end{lemma}

	\begin{proof}
		The proof is standard for rings with local units.
	\end{proof}

	The above lemma lets us assume, without loss of generality, that our set of local units $E$ is closed under join and meets. For any $E' \subseteq E$ that is also closed under join and meets, a partial order is defined on $E'$ by $e_1 \leq e_2 \Leftrightarrow e_1e_2 = e_1$. The partial order on $E'$ also makes $E'$ a directed set because $E'$ is closed under joins and meets. It's easy to prove that for $e_1 \leq e_2$, we have that $Re_1 \subseteq Re_2$ and $e_1Re_1 \subseteq e_2Re_2$. Hence, we have natural directed limits of $R$-modules and rings with morphisms of inclusion.
	
	Our next lemma calculates several direct limits of submodules and subrings of $R$ and shows that they nicely evaluate to submodules and subrings of $R$ as well.
	
	\begin{lemma} \label{arbdirectlimit} Let $E' \subseteq E$ be closed under joins and meets. Then,
		\[\varinjlim\limits_{e \in E'} Re \cong \bigcup_{e \in E'} Re\] as an $R$-module and \[\varinjlim\limits_{e \in E'} eRe \cong \bigcup_{e \in E'} eRe\] as a ring with local units.
		
		In particular, if $E' = E$, then $\varinjlim\limits_{e \in E'} eRe \cong R$.
	\end{lemma}
	\begin{proof}
		Both proofs are similar, so we just prove that \[\varinjlim\limits_{e \in E'} Re \cong \bigcup_{e \in E'} Re\] as an $R$-module. Let our directed limit be \[(\{Re_i\}_{e_i \in E'}, \{\psi_{ij}\}_{e_i \leq e_j})\] where $\psi_{ij}: Re_i \hookrightarrow Re_j$ is the injection. 
		
		We first prove that $\bigcup_{e \in E'} Re$ is an $R$-module. It is clearly closed under scalar multiplication. To prove that it is closed under addition, consider $r_1e_1 + r_2e_2$ for $r_1, r_2 \in R$ and $e_1, e_2 \in E'$. We have that $e_1 \vee e_2 \in E'$ by assumption, hence, $r_1e_1 + r_2e_2 = (r_1e_1 + r_2e_2) (e_1 \vee e_2) \in R(e_1 \vee e_2)$ so we have that our union is closed under addition as well.
		
		For any $e_i \in E'$, we clearly have an $R$-module morphism $\varphi_{e_i}: Re_i \rightarrow \bigcup _{e \in E'} Re$ for any $e_i \in E'$ and it satisfies the morphisms of our directed limits by construction. It remains to prove that $\bigcup _{e \in E'} Re$ with the maps $\varphi_e: Re \rightarrow \bigcup _{e \in E'} Re$ satisfy the universal property of the directed limit. 
		
		For $e_i \leq e_j \in E'$, let $\psi_{ij}$ denote the injection $Re_i \hookrightarrow Re_j$. Now let $M$ be another $R$-module that has maps $\{\varphi'_{e_i}: Re_i \rightarrow M\}_{e_i \in E'}$ from our diagram. If there is a map $f: \bigcup_{e \in E'} Re \rightarrow M$ that makes our diagram commute, it must be unique because every element in $\bigcup_{e \in E'} Re$ is in the image of $\varphi_{e_i}$ for some $e_i \in E$. 
		
		To prove that the morphism $f$ exists, for every $re \in \bigcup_{e \in E'} Re$, we define $f(re) = \varphi_e(re)$. This is clearly a valid homomorphism if it is consistent, so it remains to show that the definition is irrespective of representation. Consider $r_1e_1 = r_2e_2$ with their respective maps $\varphi_{e_i}: Re_i \rightarrow M$. We want to show that $\varphi_{e_1}(r_1e_1) = \varphi_{e_2}(r_2e_2)$. 
		
		Note that $r_1e_1 = r_2e_2$ implies $r_1e_1 = r_1e_1e_2 = r_2e_1e_2 = r_2e_2 \in R(e_1 \wedge e_2)$. Define $e_3 \coloneqq e_1 e_2$ which is in $E'$ because our local units are closed under meets. By the commutativity of our diagram, we have that $\varphi_{e_1}(r_1e_1) = \varphi_{e_3}(r_1(e_1e_2)) = \varphi_{e_2}(r_2e_2)$, so we are done.
		
		To prove the last statement, note that if $E = E'$, then we have that $\bigcup_{e \in E'} eRe = R$ by definition of local units. Hence,  \[\varinjlim\limits_{e \in E'} eRe \cong R\]
	\end{proof}
	
	Now fix a set of local units $E$ and a subset $E' \subseteq E$ closed under joins and meets. Our next lemma constructs a compatible set (Definition~\ref{compatibleset}) that will be used to build a locally projective module to apply Theorem~\ref{localequiv}.
	
	\begin{lemma} \label{constructcompat}
		$\{\{Re_i\}_{e_i \in E'}, \{\varphi_{ij}\}_{e_i \leq e_j}, \{\psi_{ji}\}_{e_i \leq e_j}\}$ is a compatible set in $R$-mod of finitely generated projective modules with $\varphi_{ij}: Re_i \hookrightarrow Re_j$  as injection and $\psi_{ji}: Re_j \mapsto Re_i$ as defined by $m \mapsto me_i$.
	\end{lemma}
	
	\begin{proof}
		The checks are all easy and done in \cite[Section~10]{Firrisa2020MoritaEO} for their specific case.
	\end{proof}
	
	Using Remark~\ref{endodirect} and the previous lemma, we construct a directed limit \[\langle \{\mathrm{End}_R(Re_i)^{\mathrm{op}}\}_{e_i \in E'}, \{\phi_{ij}\}_{e_i \leq e_j} \rangle\] The following lemma computes this direct limit. 
	
	\begin{theorem} \label{homdirlimit}
		Let $h_{ij}: e_iRe_i \hookrightarrow e_jRe_j$ for $e_i \leq e_j$ denote the injection.  Then,
		\[\left\langle \{\mathrm{End}_R(Re_i)^{\mathrm{op}}\}_{e_i \in E'}, \{\phi_{ij}\}_{e_i \leq e_j} \right\rangle \cong \left\langle \{e_i R e_i\}_{e_i \in E'}, \{h_{ij}\}_{e_i \leq e_j}\right\rangle\]
		as directed limits. Namely, \[\varinjlim\limits_{e_i \in E'} \mathrm{End}_R(Re_i)^{\mathrm{op}} \cong \bigcup_{e_i \in E'} e_i R e_i\] as rings with local units.
	\end{theorem}
	\begin{proof}
		It's well known that for any idempotent $e \in E$ there is a natural ring isomorphism from \[\mathrm{End}(Re)^{\mathrm{op}} \xrightarrow{\sim} eRe\]  \[f \mapsto f(e)\] We just have to prove that the morphism (and its inverse) commutes under the direct limit on $E'$. Namely, the following diagram must commute for $e_i \leq e_j \in E'$.

		\begin{center}
			\begin{tikzcd}
				\mathrm{End}(Re_j)^{\mathrm{op}} \arrow[r, "f \mapsto f(e_j)"] & e_jRe_j \\
				\mathrm{End}(Re_i)^{\mathrm{op}} \arrow[r, "f \mapsto f(e_i)", swap] \arrow[u, "f \mapsto \varphi_{ij} \circ f \circ \psi_{ji}"] & e_i R e_i \arrow[u, hook, "h_{ij}", swap] \\
			\end{tikzcd}
		\end{center}
		
		Let $f \in \mathrm{End}(Re_i)^{\mathrm{op}}$. We have that $(\varphi_{ij} \circ f \circ \psi_{ji})(e_j) = \varphi_{ij}\circ f (e_i) = f(e_i) \in e_j R e_j$. In the other direction, we have that $f \mapsto f(e_i)$ which is then injected into $e_j R e_j$. Hence, in both cases, we get that the image of $f$ is the image of $f(e_i)$ in $e_j R e_j$. Proving the reverse direction also commutes is very similar, so we omit it. 
		
		This proves that the directed limits are isomorphic. Hence, applying Theorem~\ref{arbdirectlimit} \[\varinjlim\limits_{e_i \in E'} \mathrm{End}_R(Re_i)^{\mathrm{op}} \cong \varinjlim\limits_{e_i \in E'} e_i R e_i \cong  \bigcup_{e_i \in E'} e_iRe_i\] as rings with local units.
		
	\end{proof}
	
	With the notation now set up, we now step through the proofs in \cite{abrams2008leavitt} and \cite{Firrisa2020MoritaEO} and derive some general conditions which guarantee that a subring $S \subseteq R$ satisfies $S \sim_M R$. Fix a set of local units $E$ of $R$ that is closed under meets and joins.
	
	\begin{condition} \label{localheredassum}
		The set $E \cap S$ is a set of local units for $S$.
	\end{condition}
	
	Note that $E \cap S$ is a subset of the local units $E$ that is closed under the join and meet operations because $S$ is a subring and the join and meet operations are just constructed by addition and multiplication. Hence, defining $E' \coloneqq E \cap S$, we have that $E'$ is a subset of $E$ that is closed under joins and meets.
	
	From the injection $S \hookrightarrow R$ and any $e \in E \cap S$, we can construct the injection $i_e: eSe \hookrightarrow eRe$. 
	
	\begin{condition} \label{isurjassum}
		For all $e \in E \cap S$, $i_e$ is surjective (and thus an isomorphism).
	\end{condition}
	
	Consider the module \[M = \varinjlim\limits_{e \in E \cap S} Re = \bigcup_{e \in E \cap S} Re\] which is a locally projective module with the compatible set $\{\{Re_i\}_{e_i \in E \cap S}, \{\varphi_{ij}\}_{e_i \leq e_j}, \{\psi_{ji}\}_{e_i \leq e_j}\}$ (from Lemma~\ref{constructcompat}) essentially by definition.
	
	Because $E \cap S$ is closed under the join and meet operations, we have that \[\varinjlim\limits_{e \in E \cap S} \left(\mathrm{End}_R Re \right)^{op} \cong \varinjlim\limits_{e \in E \cap S} eRe \cong \varinjlim\limits_{e \in E \cap S} eSe \cong S\] where the first isomorphism is by Theorem~\ref{homdirlimit}, the second isomorphism is by the assumption that $i_e$ is an isomorphism for all $e \in E \cap S$, and the third isomorphism is by Condition~\ref{localheredassum} and Theorem~\ref{arbdirectlimit}. This is the isomorphism required by Theorem~\ref{localequiv}.
	
	Thus, to apply Theorem~\ref{localequiv} it remains to derive a condition for when $M$ is a generator for $R$-mod. 
	
	\begin{condition} \label{generatorassum}
		For any $e \in E$, there exists some collection $\{e_i\}_{i \in I}$ where $e_i \in E \cap S$ and a surjective $R$-module homomorphism $\bigoplus_{i \in I} Re_i \rightarrow Re$.
	\end{condition}
	
	By Theorem~\ref{simpgenerator}, it suffices to show that, given Condition~\ref{generatorassum}, there is a surjective morphism from some direct sum of $M$ into $R$.
	
	We know that $R = \bigcup_{e \in E} Re$ by the definition of a local unit. Hence, if for every $e \in E$ there exists some indexing set $I$ and a surjective morphism \[\bigoplus\limits_{i \in I} M \rightarrow Re\] we can form the larger surjective morphism \[\bigoplus\limits_{(e, i) \colon e \in E, i \in I}M \rightarrow \bigoplus\limits_{e \in E} Re\] There is an obvious surjective morphism from $\bigoplus_{e \in E} Re \rightarrow \bigcup_{e \in E} Re = R$ by just adding, so we would be done.
	
	It thus suffices to show the existence of an indexing set $I$ and a surjective morphism $\bigoplus_{i \in I} M \rightarrow Re$ for any $e \in E$. To do this, note that, by Condition~\ref{generatorassum}, for any $e \in E$ there exists a set $I$ and a collection $\{e_i\}_{i \in I}$ with $e_i \in E \cap S$ such that there is a surjective homomorphism $\bigoplus_{i \in I} Re_i \rightarrow Re$. For $e_i \in E \cap S$, we have that $Me_i \cong (\bigcup_{e \in E \cap S} Re)e_i = Re_i$ so there is a surjective $R$-module morphism $M \rightarrow Re_i$. Hence, \[\bigoplus\limits_{i \in I} M \xrightarrow{\times \{e_i\}_{i \in I}}  \bigoplus\limits_{i \in I} Re_i \rightarrow Re\] is a surjective $R$-module morphism onto $Re$.
	
	Applying Theorem~\ref{localequiv}, we have the following theorem.
	
	\begin{theorem} \label{thm:moritaring}
		Let $S \subseteq R$ be a subring of a ring with local units. Let $E$ be a set of local units of $R$ that is closed under joins and meets. If the following holds, then $S \sim_M R$.
		
		\begin{enumerate}
			\item (Condition~\ref{localheredassum}) $E \cap S$ is a set of local units for $S$
			\item (Condition~\ref{isurjassum}) $eSe = eRe$ for all $e \in E \cap S$.
			\item (Condition~\ref{generatorassum}) For any $e \in E$, there exists some indexing set $I$ and a collection $\{e_i\}_{i \in I}$ with $e_i \in E \cap S$ and a surjective $R$-module homomorphism $\bigoplus_{i \in I} Re_i \rightarrow Re$.
		\end{enumerate}
	\end{theorem}

	\section{Morita Equivalence of Partial Subactions on Generalized Boolean Algebras}\label{booleanaction}
	
	Throughout this section, let $(\mathcal B_1, \Phi_1) \subseteq (\mathcal B_2, \Phi_2)$ be partial subaction of a group $G$ on generalized Boolean algebras. For $g \in G$, we will use $\mathcal I_{1, g}, \phi_{1, g}$ and $\mathcal I_{2, g}, \phi_{2, g}$ to refer to the ideals and morphisms in the partial actions on $\mathcal B_1$ and $\mathcal B_2$ respectively. In this section, we develop conditions for when $(\mathcal B_1, \Phi_1) \subseteq (\mathcal B_2, \Phi_2)$ induces Morita equivalent subalgebras $\mathrm{Lc}(R, \mathcal B_1) \rtimes_{\Phi_1} G \subseteq \mathrm{Lc}(R, \mathcal B_2) \rtimes_{\Phi_2} G$ by deriving analogs of the conditions in Theorem~\ref{thm:moritaring}. 
	
	\begin{lemma}
	Let $E = \{U \delta_e\}_{U \in \mathcal B_2}$. Then, $E$ is a set of local units for $\mathrm{Lc}(R, \mathcal B_2) \rtimes_{\Phi_2} G$ closed under joins and meets. Furthermore, $E \cap \mathrm{Lc}(R, \mathcal B_1) \rtimes_{\Phi_1} G = \{U \delta_e\}_{U \in \mathcal B_1}$ and, taking this $E$, Condition~\ref{localheredassum} is satisfied for $\mathrm{Lc}(R, \mathcal B_1) \rtimes_{\Phi_1} G \subseteq \mathrm{Lc}(R, \mathcal B_2) \rtimes_{\Phi_2} G$.
	\end{lemma}

	\begin{proof}
		By \cite[Theorem~2.28]{zhang2025partialactionsgeneralizedboolean}, $E$ is a set the set of local units for $\mathrm{Lc}(R, \mathcal B_2) \rtimes_{\Phi_2} G$ closed under joins and meets. By the identification in Lemma~\ref{lemma:ident}, the intersection of $\{U \delta_e\}_{U \in \mathcal B_2}$ with $\mathrm{Lc}(R, \mathcal B_1) \rtimes_{\Phi_1} G$ is exactly $\{U \delta_e\}_{U \in \mathcal B_1}$ which again by \cite[Theorem~2.27]{zhang2025partialactionsgeneralizedboolean} is a set of local units for $\mathrm{Lc}(R, \mathcal B_1) \rtimes_{\Phi_1} G$ so Condition~\ref{localheredassum} is satisfied.
	\end{proof}

	From now on we fix this set of local units $\{U \delta_e\}_{U \in \mathcal B_2}$. We now establish our first new condition, which has no analogous condition in Section~\ref{moritaequivsect}. However, assuming it will allow us to simplify other conditions and its proof will generally have a different flavor in our applications.
	
	\begin{condition} \label{heredassump}
		$\mathcal B_1 \subseteq \mathcal B_2$ is an ideal.
	\end{condition}
	
	In some sense, this assumption can be viewed as an analog for the hereditary condition for $C^{\ast}$-subalgebras which also has applications to Morita equivalence as shown by Rieffel in \cite{rieffel1982morita}. Henceforth, we informally refer to this condition as the ideal condition.
	
	We now derive some analogs for Condition \ref{isurjassum}.	For ease, we repeat the condition below in our context.

	\begin{remark}\label{surjassump} Condition~\ref{isurjassum} is equivalent to:
		
		For all $U \in \mathcal B_1$, the injective map \[(U \delta_e) (\mathrm{Lc}(R, \mathcal B_1) \rtimes_{\Phi_1} G) (U \delta_e) \hookrightarrow (U \delta_e) (\mathrm{Lc}(R, \mathcal B_2) \rtimes_{\Phi_2} G)  (U \delta_e)\] is also surjective.
	\end{remark}
	
	\begin{lemma}\label{firstisurj}
		Condition~\ref{isurjassum} holds if for all $g \in G$ we have that $U \cap \phi_{2, g}(U \cap V) \in \mathcal I_{1, g}$ for all $U \in \mathcal B_1$ and $V \in \mathcal I_{2, g^{-1}}$.
	\end{lemma}
	\begin{proof}
		Let $U \in \mathcal B_1$ be arbitrary.
		
		Because $U \delta_e$ is an idempotent, it suffices to prove the map
		\[(\mathrm{Lc}(R, \mathcal B_1) \rtimes_{\Phi_1} G)  \rightarrow (U \delta_e) (\mathrm{Lc}(R, \mathcal B_2) \rtimes_{\Phi_2} G)  (U \delta_e)\] is surjective. The set $\{V \delta_g\}_{g \in G, V \in \mathcal I_{2, g}}$ $R$-spans $\mathrm{Lc}(R, \mathcal B_2) \rtimes_{\Phi_2} G$, so to prove surjectivity we simply need to prove that the set $\{(U \delta_e)(V \delta_g)(U\delta_e)\}_{g \in G, V \in \mathcal I_{2, g}}$ is in the image. 
		
		Fixing arbitrary $g \in G$ and $V \in \mathcal I_{2, g^{-1}}$, we find that \[(U \delta_e)(V \delta_g)(U\delta_e) = (U \cap \phi_{2, g}(U \cap V)) \delta_g\] By Condition~\ref{isurjassum}, we have that $U \cap \phi_{2, g}(U \cap V) \in \mathcal I_{1, g}$, so we have that $(U \delta_e)(V \delta_g)(U\delta_e) \in \mathrm{Lc}(R, \mathcal B_1) \rtimes_{\Phi_1} G$, so we are done.
	\end{proof}
	
	Assuming the ideal condition is already proven, we will prove an easier condition.
	
	\begin{lemma}\label{secondisurj}
		Condition~\ref{isurjassum} holds if the ideal condition holds and for all $U \in \mathcal B_1$ and $V \in \mathcal I_{2, g^{-1}}$,  there exists some $Z \in \mathcal I_{1, g}$ such that $U \cap \phi_{2, g}(U \cap V) \leq Z$
	\end{lemma}
	
	\begin{proof}
		If $\mathcal B_1 \subseteq \mathcal B_2$ is an ideal, it is easy to show that $\mathcal I_{1, g} \subseteq \mathcal B_2$ is also an ideal. Hence, if $U \cap \phi_{2, g}(U \cap V) \leq Z$ for some $Z \in \mathcal I_{1, g}$, then $U \cap \phi_{2, g}(U \cap V) \in \mathcal I_{1, g}$. This is the condition for Lemma~\ref{firstisurj}.
	\end{proof}
	
	Our sets usually have covers, and the proof often becomes easier when we restrict our attention to covers. Hence, we present our final condition using the language of covers.
	
	\begin{condition}\label{thirdisurj}
		Let $C_1 \subseteq \mathcal B_1$ be a cover of $\mathcal B_1$. For all $g \in G$, let $C_g \subseteq \mathcal I_{2, g}$ be a cover of $\mathcal I_{2, g}$. For all $g \in G$, $X, Y \in C_1$ and $V \in C_{g^{-1}}$, there exists some $Z \in \mathcal I_{1, g}$ such that $X \cap \phi_{2, g}(Y \cap V) \leq Z$.
	\end{condition}
	\begin{theorem}
		If Condition~\ref{heredassump} and Condition~\ref{thirdisurj} hold, then Condition~\ref{isurjassum} holds.
	\end{theorem}
	\begin{proof}		
		Let $U \in \mathcal B_1$ and $V \in \mathcal I_{2, g^{-1}}$ be arbitrary as in Lemma~\ref{secondisurj}. Because $C_1$ is a cover of $\mathcal B_1$, we have $U \subseteq \bigcup_{i=1}^n U_i$ for $U_i \in C_1$ and $V \subseteq \bigcup_{i=1}^m V_i$ for $V_i \in C_1$. We find that \[U \cap \phi_{2, g}(U \cap V) \leq \left(\bigcup_{i=1}^n U_i\right) \cap \phi_{2, g}(\left(\bigcup_{i=1}^n U_i\right) \cap \left( \bigcup_{i=1}^m V_i \right)) \] so it suffices to prove that there exists some $Z \in \mathcal I_{1, g}$ where the right-hand side is less than $Z$. 
		
		We calculate that \[\left(\bigcup_{i=1}^n U_i\right) \cap \phi_{2, g}(\left(\bigcup_{i=1}^n U_i\right) \cap \left( \bigcup_{i=1}^m V_i \right)) =  \bigcup_{i, j=1}^n  \bigcup_{k=1}^m (U_i \cap \phi_{2, g}(U_j \cap V_k))\] By the assumption, each of these terms is $\leq Z_{ijk}$ for some $Z_{ijk} \in \mathcal I_{1, g}$. Taking $Z \coloneqq \bigcup_{i, j = 1}^n\bigcup_{k=1}^m Z_{ijk} \in \mathcal I_{1, g}$, we have inequality required in Lemma~\ref{secondisurj}.
		
	\end{proof}
	
	We now derive analogs for Condition \ref{generatorassum}. Again, we repeat the condition in our current context.
	
	\begin{remark}\label{algebrageneratorassum} Condition~\ref{generatorassum} is equivalent to:
		
		For any $U \in \mathcal B_2$, there exists collection $\{U_i\}_{i \in I}$ with $U_i \in \mathcal B_1$ and a surjective $(\mathrm{Lc}(R, \mathcal B_2) \rtimes_{\Phi_2} G)$-module homomorphism \[\bigoplus_{i \in I} (\mathrm{Lc}(R, \mathcal B_2) \rtimes_{\Phi_2} G)(U_i \delta_e) \rightarrow (\mathrm{Lc}(R, \mathcal B_2) \rtimes_{\Phi_2} G)(U \delta_e)\]
	\end{remark}
	
	\begin{lemma} \label{covershifts}
		Condition~\ref{generatorassum} holds if the ideal condition holds and for all $U \in \mathcal B_2$, we have that $U \leq \bigcup_{i=1}^n \phi_{2, g_i}(V_i)$ for some $g_i \in G$ and $V_i \in \mathcal I_{2, g^{-1}}\cap \mathcal B_1$.
	\end{lemma}

	\begin{proof}
		Let $M \coloneqq \mathrm{Lc}(R, \mathcal B_2) \rtimes_{\Phi} G$. We first show that we can assume that \[U = \bigcup_{i=1}^n \phi_{2, g_i}(V_i)\] Let $U' \coloneqq \bigcup_{i=1}^n \phi_{2, g_i}(V_i)$ and $U' \delta_e$ it's associated local unit. Note that we have a surjective morphism $M(U'\delta_e) \rightarrow M(U\delta_e)$ by right multiplication by $U\delta_e$. Hence, if $U \lneq U'$, we would still be able to construct a surjective morphism to $M(U \delta_e)$. Thus, from now on we assume that $U = \bigcup_{i=1}^n \phi_{2, g_i}(V_i)$. 
		
		We now show that we can also assume that $\{\phi_{2, g_i}(V_i)\}_{i = 1}^n$ are pairwise disjoint.  Because $\phi_{2, g_i}$ is an isomorphism, there clearly exists a set $\{V'_i\}_{i=1}^n$ with $V'_i \in \mathcal I_{2, g^{-1}}$ where $\{\phi_{2, g_i}(V'_i)\}_{i=1}^n$ are pairwise disjoint and $\bigcup_{i=1}^n \phi_{2, g_i}(V_i) = \bigcup_{i=1}^n \phi_{2, g_i}(V'_i)$. Furthermore, it's easy to construct these sets such that $V'_i \subseteq V_i$. Thus, $V'_i \in \mathcal B_1$ as well because $\mathcal B_1 \subseteq \mathcal B_2$ is an ideal, so we find that $V'_i \in \mathcal I_{2, g^{-1}} \cap \mathcal B_1$. Hence, this is still a valid representation.
		
		Now let $x \coloneqq U \delta_e$ and $x_i \coloneqq \phi_{2, g_i}(V_i) \delta_e$ with the above conditions. We now construct a surjective homomorphism \[\bigoplus_{i=1}^n Mx_i \rightarrow Mx\] Because $U = \bigsqcup_{i=1}^n \phi_{2, g_i}(V_i)$, we have the equality $x = \sum_{i=1}^n x_i$. There is a morphism from $\bigoplus_{i=1}^n Mx_i \rightarrow Mx$ given by $(a_i)_{i=1}^n \mapsto \sum_{i=1}^n a_i$ because $x_ix = x_i$ and it is surjective because $\sum_{i=1}^n x_i = x$.
		
		Let $p_i \coloneqq V_i \delta_e$. We will show that $Mx_i \cong Mp_i$, which will suffice because $V_i \in \mathcal B_1$ and we would have a surjection \[\bigoplus_{i=1}^n Mp_i \rightarrow Mx\] To do this, we define $b_i \coloneqq \phi_{2, g_i}(V_i) \delta_{g_i}$ and $b_i^{\ast} \coloneqq V_i\delta_{g_i^{-1}}$. 
		
		We have the following calculations \begin{enumerate}
			\item $b_i b_i^{\ast} = \phi_{2, g_i}((\phi_{2, g_i^{-1}}(\phi_{2, g_i}(V_i)) \cap V_i)) \delta_e = \phi_{2, g_i}(V_i)\delta_e = x_i$
			\item $b_i p_i = \phi_{2, g}(\phi_{2, g_i^{-1}}(\phi_{2, g_i}(V_i)) \cap V_i) \delta_{g_i} = \phi_{2, g}(V_i) \delta_{g_i} = b_i$
						\item $b_i^{\ast}b_i = \phi_{2, g_i^{-1}}(\phi_{2, g_i}(V_i) \cap \phi_{2, g_i}(V_i)) \delta_e = V_i \delta_e = p_i$

		\end{enumerate}

		There is a morphism $Mx_i \xrightarrow{\times b_i} Mb_i$ and, using (2), we have that $b_i p_i = b_i$ so $Mb_i \subseteq Mp_i$ so this can be extended to a morphism \[Mx_i \xrightarrow{\times b_i} Mp_i\] Similarly, there is a map $Mp_i \xrightarrow{\times b^{\ast}_i} Mp_i b^{\ast}_i$. Using (1), we find that $p_i = b_i^{\ast}b_i$ so we can calculate $Mp_ib^{\ast}_i = M b_i^{\ast}b_i b_i^{\ast}$. Applying (3), we find that $M b_i^{\ast}b_i b_i^{\ast} = M b_i^{\ast} x_i \subseteq Mx_i$ which gives us a map
		
		\[Mp_i \xrightarrow{\times b_i^{\ast}} Mx_i\] We will show that these maps are inverses to each other, which suffices to prove that they are isomorphisms.
		
		By (1) and the fact that $x_i$ is idempotent, the sequence \[Mx_i \xrightarrow{\times b_i} Mp_i\xrightarrow{\times b_i^{\ast}} Mx_i\] is the identity.  By (3), and the fact that $p_i$ is idempotent, the sequence \[Mp_i\xrightarrow{\times b_i^{\ast}} Mx_i\xrightarrow{\times b_i} Mp_i\] is also the identity, so we are done
	\end{proof}
	
	Starting with the next lemma, we build up to a theorem that is simpler to check in practice.
	
	\begin{lemma} \label{twosided}
		If the ideal condition is satisfied and $\mathcal I_{1, g} = \mathcal B_1 \cap \mathcal I_{2, g}$ for all $g \in G$. Then the subalgebra $\mathrm{Lc}(R, \mathcal B_1) \rtimes_{\Phi_1} G \subseteq \mathrm{Lc}(R, \mathcal B_2)	\rtimes_{\Phi_2} G$ is also a two-sided ideal.
	\end{lemma}
	
	\begin{proof}
		Because $\mathrm{Lc}(R, \mathcal B_1)\rtimes_{\Phi_1} G$ has a set of local units $\{U \delta_e \}_{U \in \mathcal B_1}$, it suffices to prove that the local units are closed under left and right multiplication by elements in $\mathrm{Lc}(R, \mathcal B_2) \rtimes_{\Phi_2} G$. Because $\mathrm{Lc}(R, \mathcal B_2)	\rtimes_{\Phi_2} G$ is $R$-spanned by elements of the form $V \delta_g$ for $V \in \mathcal I_{2, g}$, we can restrict our attention to multiplication by these elements.
		
		Let $U \in \mathcal B_1$ and $g \in G$ with $V \in \mathcal I_{2, g}$ be arbitrary. We calculate that \[(V \delta_g)(U\delta_e) = \phi_{2, g}(\phi_{2, g^{-1}}(V) \cap U)\delta_{g}\] Because $U \in \mathcal B_1$ and $\mathcal B_1 \subseteq \mathcal B_2$ is an ideal, we know that $\phi_{2, g^{-1}}(V) \cap U \in \mathcal B_1$. Furthermore, $\phi_{2, g^{-1}}(V) \in \mathcal I_{2, g}$ and $\mathcal I_{1, g} = \mathcal I_{2, g} \cap \mathcal B_2$, so we deduce that $\phi_{2, g^{-1}}(V) \cap U \in \mathcal I_{1, g}$. Hence, we find that $\phi_{2, g}(\phi_{2, g^{-1}}(V) \cap U)\delta_g = \phi_{1, g}(\phi_{2, g^{-1}}(V) \cap U)\delta_g$ which is in $\mathrm{Lc}(R, \mathcal B_1) \rtimes_{\Phi_1} G$.

		For multiplication on the other side, we calculate that $(U \delta_e)(V \delta_g) = (U \cap V)\delta_g$. Because $\mathcal B_1 \subseteq \mathcal B_2$ is an ideal we have that $U \cap V \in \mathcal B_1$. Because $V \in \mathcal I_{2, g}$, we have that $U \cap V \in \mathcal I_{2, g}$ so $U \cap V \in \mathcal B_1 \cap \mathcal I_{2, g} = \mathcal I_{1, g}$, so $(U \cap V) \delta_g \in \mathrm{Lc}(R, \mathcal B_1) \rtimes_{\Phi_1} G$, and we are done.
		
	\end{proof}
	
	Our main goal is to construct an ideal of $\mathrm{Lc}(R, \mathcal B_2)	\rtimes_{\Phi_2} G$ that contains the subalgebra $\mathrm{Lc}(R, \mathcal B_1) \rtimes_{\Phi_1} G$. With the last lemma in mind, one obvious attempt at this would be to take the partial action on $\mathcal B_1$ with ideals $\mathcal B_1 \cap \mathcal I_{2, g}$ and actions induced from $\phi_{2, g}$. However, the issue with taking that particular partial action is that there may be some $U \in \mathcal B_1 \cap \mathcal I_{2, g^{-1}}$ such that $\phi_{2, g}(U) \in \mathcal B_2 \setminus \mathcal B_1$. To rectify this issue, we construct an intermediate generalized Boolean algebra $\mathcal B$ that essentially closes $\mathcal B_1$ on $\phi_{2, g}$ operations.
	
	\begin{lemma} \label{intermediatesystem}
		Let $\mathcal B_1 \subseteq \mathcal B_2$ be an ideal. Then, \[\mathcal B \coloneqq \left\{\bigcup_{i=1}^n \phi_{2, g_i}(V_i) \in \mathcal B_2  \colon n \geq 1 \text{ for some  } g_i \in G \text{ and }  V_i \in \mathcal B_1 \cap \mathcal I_{2, g_i^{-1}}\right\}\] is an intermediate generalized Boolean algebra  that sits between $\mathcal B_1 \subseteq \mathcal B_2$. Furthermore, $\mathcal B \subseteq \mathcal B_2$ is an ideal.
		
		The ideals $\mathcal I_g = \mathcal B \cap \mathcal I_{2, g}$ sitting between $\mathcal I_{1, g} \subseteq \mathcal I_{2, g}$ with valid induced morphisms $\phi_g: \mathcal I_{g^{-1}} \rightarrow \mathcal I_g$ as the restriction of $\phi_{2, g}$ on $\mathcal I_g$ form a partial action on $\mathcal B$ denoted as $\Phi$. This gives us a sequence of partial actions on generalized Boolean algebras $(\mathcal B_1, \Phi_1) \subseteq (\mathcal B, \Phi) \subseteq  (\mathcal B_2, \Phi_2)$.
	\end{lemma}
	
	\begin{proof}
		We first want to prove that $\mathcal B$ is a generalized Boolean algebra that sits between $\mathcal B_1 \subseteq \mathcal B_2$. Note that $\mathcal B$ is obviously closed under unions, contains $\mathcal B_1$ (using $n = 1$ and $g_1 = e$) and is contained in $\mathcal B_2$. To prove that it is closed under relative complements and intersections, it suffices to prove that $\mathcal B \subseteq \mathcal B_2$ is an ideal. Fixing some set in $\mathcal B$, let $U \in \mathcal B_2$ be such that $U \leq \bigcup_{i=1}^n \phi_{2, g_i}(V_i)$ for $V_i \in \mathcal B_1 \cap \mathcal I_{2, g^{-1}}$. Taking intersection by $U$ on both sides, we find that \[U = \bigcup_{i=1}^n (\phi_{2, g_i}(V_i) \cap U)\] By similar reasoning to the proof of Lemma~\ref{covershifts} we find that there is some $V'_i \in \mathcal B_2$ such that $V'_i \leq V_i \in \mathcal B_1 \cap \mathcal I_{2, g^{-1}}$ and $\phi_{2, g_i}(V'_i) = \phi_{2, g_i}(V_i) \cap U$. Because $\mathcal B_1 \subseteq \mathcal B_2$ is an ideal, we have that $\mathcal B_1 \cap \mathcal I_{2, g^{-1}}$ is an ideal of $\mathcal B_2$ as well. Hence, $V'_i \in \mathcal B_1 \cap  \mathcal I_{2, g^{-1}}$ so, by definition, $U \in \mathcal B$.
		
		We now need to prove that $\phi_{2, g}$ is well-defined as a restriction to a map $\mathcal I_{g^{-1}} \rightarrow \mathcal I_g$. Consider an arbitrary element $\bigcup_{i=1}^n \phi_{2, g_i}(V_i) \in \mathcal I_{g^{-1}} = \mathcal B \cap \mathcal I_{2, g^{-1}}$. Note that this implies that $\phi_{2, g_i}(V_i) \in \mathcal I_{2, g^{-1}}$. Hence, we get that \[\phi_{2, g}\left(\bigcup_{i=1}^n \phi_{2, g_i}(V_i)\right) = \bigcup_{i=1}^n \phi_{2, g}(\phi_{2, g_i}(V_i)) = \bigcup_{i=1}^n \phi_{2, gg_i}(V_i)\] This is in $\mathcal B$ by definition and also in $\mathcal I_{2, g}$ trivially so the element is in $\mathcal B \cap \mathcal I_{2, g} = \mathcal I_g$.
		
		Because the morphism $\phi_g$ is the restriction of $\phi_{2, g}$ to $\mathcal I_{g^{-1}}$ and $\phi_{1, g}$ is the further restriction to $\mathcal I_{1, g^{-1}}$, we get a sequence of partial actions on generalized Boolean algebras. $(\mathcal B_1, \Phi_1) \subseteq (\mathcal B, \Phi) \subseteq  (\mathcal B_2, \Phi_2)$.
	\end{proof}
	
	\begin{condition} \label{twoidealcond}
	 The subalgebra $\mathrm{Lc}(R, \mathcal B_1) \rtimes_{\Phi_1} G \subseteq  \mathrm{Lc}(R, \mathcal B_2) \rtimes_{\Phi_2} G$ is contained in no proper two-sided ideal.
	\end{condition}
	\begin{theorem} \label{twosidedideal}
		If Condition~\ref{heredassump} and Condition~\ref{twoidealcond} hold, then Condition~\ref{generatorassum} holds.
	\end{theorem}
	
	\begin{proof}
		Build the intermediate partial action $(\mathcal B, \Phi)$ from Definition~\ref{intermediatesystem} and realize the sequence of $R$-subalgebras \[\mathrm{Lc}(R, \mathcal B_1) \rtimes_{\Phi_1} G  \subseteq \mathrm{Lc}(\mathcal B, R) \rtimes_{\Phi} G \subseteq \mathrm{Lc}(R, \mathcal B_2) \rtimes_{\Phi_2} G\]
		
		Assume that Condition~\ref{generatorassum} does not hold. Then, by the contrapositive of Lemma~\ref{covershifts} and construction of $\mathcal B$, we have that $\mathcal B \subsetneq \mathcal B_2$ is a strict subset. This means that $\mathrm{Lc}(\mathcal B, R) \rtimes_{\Phi} G \subsetneq \mathrm{Lc}(R, \mathcal B_2) \rtimes_{\Phi_2} G$ is a proper subalgebra.
		
		By Lemma~\ref{twosided}, because $\mathcal I_g = \mathcal B \cap \mathcal I_{2, g}$ by construction, the proper subalgebra is also a proper two-sided ideal. Hence, the subalgebra $\mathrm{Lc}(R, \mathcal B_1) \rtimes_{\Phi_1} G$ is contained in the proper two-sided ideal $\mathrm{Lc}(\mathcal B, R) \rtimes_{\Phi} G$, so we are done.
	\end{proof} 
	
	\begin{remark}
		This condition is quite similar to the notion of a full $C^{\ast}$-subalgebra $B \subseteq A$, which is a condition that states that $B$ is not contained in any proper two-sided closed ideal of $A$. If $B \subseteq A$ is also hereditary, Rieffel proved in \cite{rieffel1982morita} that $B$ and $A$ are Morita equivalent as $C^{\ast}$-algebras.
	\end{remark}
	We now summarize our main result.
	\begin{theorem} \label{MoritaBoolean}
		Let $(\mathcal B_1, \Phi_1) \subseteq (\mathcal B_2, \Phi_2)$ be a partial subaction on generalized Boolean algebras. If the following holds, then $\mathrm{Lc}(R, \mathcal B_1) \rtimes_{\Phi_1} G \sim_M \mathrm{Lc}(R, \mathcal B_2) \rtimes_{\Phi_2} G$.
		
		\begin{enumerate}
			\item (Condition~\ref{heredassump}) $\mathcal B_1 \subseteq \mathcal B_2$ is an ideal
			\item (Condition~\ref{thirdisurj}) Fixing any covers $C_1 \subseteq \mathcal B_1$ and $C_g \subseteq \mathcal I_{2, g}$, for all $g \in G$, $X, Y \in C_1$, and $V \in C_g$, we have that $X \cap \phi_{2, g}(Y \cap V) \leq Z$ for some $Z \in \mathcal I_{1, g}$.
			\item (Condition~\ref{twoidealcond}) The subalgebra $\mathrm{Lc}(R, \mathcal B_1) \rtimes_{\Phi_1} G \subseteq \mathrm{Lc}(R, \mathcal B_2) \rtimes_{\Phi_2} G$ is contained in no proper two-sided ideal
		\end{enumerate}
	\end{theorem}

\section{Morita Equivalence of Inverse Subsemigroups} \label{inversegroupaction}

Throughout this section, we employ notation and theorems found in \cite[Section~4]{zhang2025partialactionsgeneralizedboolean}. Let $S_1$ and $S_2$ be strongly $E^{\ast}$-unitary inverse semigroups such that $S_1 \subseteq_c S_2$. Let $G$ be a group and $\varphi: S_2^{\times} \rightarrow G$ be a map with the required properties in Definition~\ref{definition:stronglyunit}. Let $(\mathcal T_c(E_1), \Phi_1)$ and $(\mathcal T_c(E_2), \Phi_2)$ be the associated partial actions of $G$ to $(S_1, \varphi_{S_1^{\times}})$ and $(S_2, \varphi)$. Because $S_1 \subseteq_c S_2$, this lets us induce a partial sub-action $(\mathcal T_c(E_1), \Phi_1) \subseteq (\mathcal T_c(E_2), \Phi_2)$. We will use the conditions in Theorem~\ref{MoritaBoolean} to derive conditions for when $S_1 \subseteq_c S_2$ induces Morita equivalent $R$-subalgebras \[L_R(S_1) = \mathrm{Lc}(R, \mathcal T_c(E_1)) \rtimes_{\Phi_1} G \subseteq \mathrm{Lc}(R, \mathcal T_c(E_2)) \rtimes_{\Phi_2} G =  L_R(S_2)\] Importantly, our final conditions will be irrespective of the group $G$ and grading $\varphi: S_2^{\times} \rightarrow G$.

\begin{condition}\label{cond1inv}
	$E_1 \subseteq E_2$ is tight.
\end{condition}
\begin{theorem}\label{semigroupcovertight}
	If Condition~\ref{cond1inv} holds, then Condition~\ref{heredassump} holds.
\end{theorem}
\begin{proof}
For $S_1 \subseteq_c S_2$, the identification of the partial subaction is given by the inclusion $\mathcal T_c(E_1) \subseteq \mathcal T_c(E_2)$ by \cite[Theorem~4.19]{zhang2025partialactionsgeneralizedboolean}. The definition of tight is exactly that $\mathcal T_c(E_1) \subseteq \mathcal T_c(E_2)$ is an ideal.
\end{proof}

We now derive an analog for Condition~\ref{thirdisurj}.

\begin{remark} Recall that by \cite[Lemma~3.6]{zhang2025partialactionsgeneralizedboolean} the set $C_1 = \{V_x\}_{x \in E_1}$ forms a cover for $\mathcal T_c(E_1)$ and $C_g = \{V_x\}_{x \in E_{2, g}}$ forms a cover for $\mathcal T_c(E_{2, g})$. With these covers in mind, Condition~\ref{thirdisurj} is equivalent to:
	
For all $g \in G$, $x, y \in E_1$, and $a \in E_{2, g^{-1}}$,  we have that $V_x \cap \phi_{2, g}(V_y \cap V_a) \leq Z$ for some $Z \in \mathcal T_c(E_{1, g})$.
\end{remark}

The next lemma gives a sufficient condition for Condition~\ref{thirdisurj}.

\begin{lemma}\label{firstsurj}
 Condition~\ref{thirdisurj} holds if for all $g \in G$, $x, y \in E_1$ and $s \in S_2$ that satisfy $\varphi(s) = g$, we have that $xs y s^{\ast} \leq z$ for some $z \in E_{1, g}$.
\end{lemma}
\begin{proof}

Let $g \in G$, $x, y \in E_1$ and $a \in E_{2, g^{-1}}$ be arbitrary. There exists some $s \in S_2$ such that $\varphi(s) = g$ and $a \leq s^{\ast}s$. Note that $ya \leq s^{\ast}s$ as well, so we know that $ya \in E_{2, g^{-1}}$ and that $\phi_{2, g}(ya) = s ya s^{\ast} \leq sys^{\ast}$. Then, note that \[V_x \cap \phi_{2, g}(V_y \cap V_a) = V_x \cap \phi_{2, g}(V_{ya}) = V_x \cap V_{syas^{\ast}} \leq V_x \cap V_{sys^{\ast}} = V_{xsys^{\ast}}\] By the assumption, there is some $z \in E_{1, g}$ such that $xsys^{\ast} \leq z$ and hence $V_{xsys^{\ast}} \leq V_z \in \mathcal T_c(E_{2, g})$, so we are done.
\end{proof}
	
We now present a condition that is often easier to check and is included in the main theorem. A benefit of this corollary is that it makes no explicit mention of $g$, which makes casework simpler in application and removes the dependence on a specific grading $\varphi$. Recall that the order of arbitrary elements in an inverse semigroup $S$ is defined as $x \leq y \Leftrightarrow x = ey$ for some idempotent $e \in E$ and that this order is consistent with the order on $E$. Furthermore, it's easy to see that if $x \leq y$ then $xx^{\ast} \leq yy^{\ast}$.

\begin{condition} \label{cond2inv}
	For all $x, y \in E_1$ and $s \in S_2$ there exists $s' \in S_1$ such that $xsy \leq s'$.
\end{condition}

\begin{remark}
	Note the similarities between the above condition and \cite[Theorem~3.2.2]{Murphy1990-vw} which states that for a $B \subseteq A$ a $C^{\ast}$-subalgebra the condition that $B \subseteq A$ is hereditary is equivalent to $bab' \in A$ for all $b, b' \in B$ and $a \in A$.
\end{remark}
\begin{theorem} \label{bridgesurj}
If Condition~\ref{cond2inv} holds then Condition~\ref{thirdisurj} holds.
\end{theorem}
\begin{proof}
We will show that the condition for Lemma~\ref{firstsurj} holds if Condition~\ref{cond2inv} holds.

Let $g \in G$, $x, y \in E_1$ and $s \in S_2$ with $\varphi(s) = g$ be as in Lemma~\ref{firstsurj}. By Condition~\ref{cond2inv}, there is some $s' \in S_1$ such that $xsy \leq s'$. Let $e \in E_2$ be the idempotent such that $xsy = es'$.

If $xsy = 0$ then clearly $(xsys^{\ast}) \leq 0 \in E_{1, g}$, so the condition in Lemma~\ref{firstsurj} holds. Otherwise, if $xsy \neq 0$ we calculate that \[\varphi(s') = \varphi(es') = \varphi(xsy) = \varphi(s) = g\] Then, it's easy to see that \[(xsy)(xsy)^{\ast} = xsyys^{\ast}x = xsys^{\ast}\] where we use the commutativity of idempotents. Thus, we find that $xsys^{\ast} \leq s'(s')^{\ast}$. Because $\varphi(s')= g$ and $s' \in S_1$, we find that $s'(s')^{\ast} \in E_{1, g}$. Thus, we can take $z = s'(s')^{\ast}$ to satisfy the condition in Lemma~\ref{firstsurj}.
\end{proof}

We do not change Condition~\ref{twoidealcond} of Theorem~\ref{MoritaBoolean} in any way. This condition is often easier to verify anyway, so we refrain from adding extra complication and instead keep the parallel with the $C^{\ast}$-algebra case.

\begin{condition}\label{cond3inv}
	$L_R(S_1) \subseteq L_R(S_2)$ is contained in no proper two-sided ideal.
\end{condition}

\begin{remark}
	A priori, the subalgebra identification $L_R(S_1) \subseteq L_R(S_2)$ depends on the chosen grading $\varphi: S_2^{\times} \rightarrow G$. However, it's not hard to see using that up to the isomorphism in \cite[Theorem~4.10]{zhang2025partialactionsgeneralizedboolean} the identification in \cite[Theorem~4.19]{zhang2025partialactionsgeneralizedboolean} is irrespective of $\varphi$. Furthermore, in our applications, one could always fix a grading and compute that $L_R(S_1) \subseteq L_R(S_2)$ is contained in no proper two-sided ideal and the proof would procede as normal, so this is only a minor technical point.
\end{remark}

We now summarize our results in the following theorems. Throughout, let $S_1$ and $S_2$ be strongly $E^{\ast}$-unitary inverse semigroups.

\begin{theorem}\label{finalmorita1}

If the following are true, then $L_R(S_1) \sim_M L_R(S_2)$.

\begin{enumerate}
	
	\item $S_1 \subseteq_c S_2$
	\item (Condition~\ref{cond1inv}) $E_1 \subseteq E_2$ is tight
	\item (Condition~\ref{cond2inv})	For all $x, y \in E_1$ and $s \in S_2$ there exists $s' \in S_1$ such that $xsy \leq s'$
	\item (Condition~\ref{cond3inv}) $L_R(S_1) \subseteq L_R(S_2)$ is contained in no proper two-sided ideal
\end{enumerate}
\end{theorem}

In some applications, we can prove something stronger about $E_1 \subseteq E_2$, which we record in the following theorem.

\begin{theorem} \label{finalmorita2}
If the following are true, then $L_R(S_1)\sim _M L_R(S_2)$.
\begin{enumerate}
	\item $S_1 \subseteq S_2$
	\item $E_1 \subseteq E_2$ is closed downwards
	\item For all $x, y \in E_1$ and $s \in S_2$ there exists $s' \in S_1$ such that $xsy \leq s'$
	\item $L_R(S_1) \subseteq L_R(S_2)$ is contained in no proper two-sided ideal
\end{enumerate}
\end{theorem}

\begin{proof}
Theorem~\ref{finalmorita1} follows directly from Theorem~\ref{MoritaBoolean} and applying the theorems developed in this section. Theorem~\ref{finalmorita2} follows from applying \cite[Corollary~3.24]{zhang2025partialactionsgeneralizedboolean} to deduce that $E_1 \subseteq E_2$ downward directed implies that it preserves finite covers and is tight.
\end{proof}

\section{Enlargement of Inverse Semigroups} \label{enlargement}

In \cite{Lawson96}, Lawson introduced the notion of an enlargement of an inverse semigroup as a special case of a, at that time, yet to be explained Morita equivalence theory of inverse semigroups. Later, in \cite{SteinbergSemiMorita}, Steinberg defined a notion of Morita equivalence for inverse semigroups and proved that enlargement was a special case.

\begin{definition}
	An inverse semigroup $T$ is said to be an enlargement of an inverse subsemigroup $S$ if $STS = S$ and $TST = T$.
\end{definition}

Using results developed in this paper we will show that for an enlargement $S \subseteq T$ where $T$ is a strongly $E^{\ast}$-unitary inverse semigroup, we have that $L_R(S) \sim_M L_R(T)$. We emphasize that this could be done in much more generality and much easier by using a theorem of Steinberg in  \cite[Theorem~4.7 and Corollary~4.8]{SteinbergSemiMorita} that showed that Morita equivalent inverse semigroups $S$ and $T$ induce Morita equivalent tight groupoids $\mathcal G_{tight}(S)$ and $\mathcal G_{tight}(T)$. We would then be able to apply a result developed simultaneously by Clark and Sims in \cite{Clark2013EquivalentGH} and Steinberg in \cite{Steinberg14Modules} to show that the resulting Steinberg algebras $A_R(\mathcal G_{tight}(S))$ and $A_R(\mathcal G_{tight}(T))$ are Morita equivalent. Because our algebras $L_R(S)$ and $L_R(T)$ are isomorphic to these Steinberg algebras by \cite[Theorem~3.14]{zhang2025partialactionsgeneralizedboolean}, we would be done. However, in many of our applications, our inverse semigroups $S_1 \subseteq S_2$ are not just enlargements, so the discussion in this section is a special case of our results. 

\begin{theorem}
	If $S \subseteq T$ is an enlargement and $T$ is strongly $E^{\ast}$-unitary, then $L_R(S) \sim_M L_R(T)$.
\end{theorem}

\begin{proof}
	We will prove each condition of Theorem~\ref{finalmorita2}. Throughout, fix some group $G$ and grading $\varphi: T^{\times} \rightarrow G$.
	
	\begin{enumerate}
		\item This is by definition.
		\item Let $x \in E(S)$ and let $x' \in E(T)$ be such that $x' \leq x$, then because $STS = S$ we have that $xx'x = x' \in S$ and clearly remains an idempotent.
		\item For $x, y \in E(S)$ and $s \in T$, because $STS = S$ we have that $xsy \in S$.
		\item We will show that for all $x \in E(T)$, we have that $x \delta_e$ is contained in any two-sided ideal containing $L_R(S)$. Because $x \delta_e$ generates a set of local units for $L_R(T)$, any two-sided ideal containing these elements must be all of $L_R(T)$, so this would suffice. 
		
		Because $TST = T$, for any $x \in T$ we have that there exists $x_1, x_3 \in T$ and $x_2 \in S$ such that $x_1x_2x_3 = x$. We then have that $(x_1 \delta_e)(x_2 \delta_e)(x_3\delta_e) = x \delta_e$ where $(x_2\delta_e) \in L_R(S)$. It's obvious that $(x_1 \delta_e)(x_2 \delta_e)(x_3\delta_e)$ is in any two-sided ideal containing $L_R(S)$, so we are done.
	\end{enumerate}
\end{proof}

\section{Applications to Leavitt Path Algebras} \label{Leavittapp}

See \cite[Section~5.1]{zhang2025partialactionsgeneralizedboolean} for a description of the inverse semigroup associated to a graph. Our main result of this section is the following.

\begin{theorem}\label{MoritaFinite}
Let $G = (G^0, G^1)$ be a graph and let $H^0$ be a hereditary subset of $G^0$ where $H^0$ is contained in no proper saturated hereditary subset of $G$. We let $H = (H^0, H^1)$ refer to the subgraph induced from the hereditary subset.

Then, $L_R(G)$ and $L_R(H)$ are Morita equivalent.
\end{theorem}

\begin{proof}
Let $S_H$ and $S_G$ be the associated strongly $E^{\ast}$-unitary inverse semigroups of $H$ and $G$ respectively. We will invoke Theorem~\ref{finalmorita2}. 

\begin{enumerate}
	
	\item Because $H$ is a subgraph of $G$, we obviously have an obvious injection from $S_H \hookrightarrow S_G$ that takes $(\alpha, \beta) \mapsto (\alpha, \beta)$ because all paths in $H$ can be viewed as paths in $G$.
	
	\item Under the above injection and using the fact that $H$ is hereditary, it's easy to see that $E_1 \subseteq E_2$ is closed downwards.
	
	\item Let $x = (p_x, p_x)$, $y = (p_y, p_y)$ be elements of $E(S_H)$ and let $s = (\alpha, \beta)$ be an element of $S_G$. Assume that $(p_x, p_x)(\alpha, \beta)(p_y, p_y) \neq 0$. Note that if $(p_x, p_x)(\alpha, \beta) \neq 0$, we must have that $\alpha$ and $p_x$ have the same source. Similarly, if $(\alpha, \beta)(p_y, p_y) \neq 0$, we must have that $\beta$ and $p_y$ have the same source. Because $p_x$ and $p_y$ are paths in $H$, we have that their sources are in $H^0$. Hence, $\alpha$ and $\beta$ must actually be in paths in $H$ because $H$ is hereditary so $s \in S_H$ and clearly $xsy \in S_H$ as well.
	
	\item Note that the map $E(S_H) \hookrightarrow E(S_G)$ takes $(v, v) \delta_e \mapsto (v, v)\delta_e$. By our isomorphism to the Leavitt Path algebras, we find that the image contains all $v \in H^0$. By the classification of two-sided ideals found in \cite[Chapter~2]{Abrams2017}, we know that any two-sided ideal must contain the vertices for a hereditary saturated subset. If the only hereditary saturated subset containing $H^0$ is all vertices $G^0$, then any two-sided ideal containing our subring $L_R(H)$ must be the entire ring $L_R(G)$, which proves the assumption.
\end{enumerate}

\end{proof}

The following corollary is a less precise version of \cite[Proposition~3.5]{ABRAMS2007753}, which proved that the Leavitt path algebras of finite directed acyclic graphs are isomorphic to the sum of a matrix ring $\oplus_{i=1}^n M_{n_i}(K)$. In \cite{ABRAMS2007753}, it's shown that one can actually compute the values $n_i$ (and $n$). Our corollary does not compute the values of $n_i$, but does allow for possibly infinite size graphs.

\begin{corollary}[Directed Acyclic Graphs]
Let $G$ be a directed acyclic graph without infinite emitters such there for every vertex $v$ there is some $N_v>0$ such that all paths starting from $v$ are of length less than $N_v$. Then $L_R(G) \sim_M \bigoplus\limits_{v \in G_{\text{sink}}} R$.
\end{corollary}

\begin{proof}
Let $H^0 = G_{\text{sink}}$.  Because $H^0$ is just sinks, it is clearly hereditary and the induced graph is just the disjoint union of the vertices. It's well known that $L_R(H) = \bigoplus_{v \in H^0} R$.

By the conditions on our graph, for any vertex $v$, we have that the induced downward hereditary directed graph is of finite height and thus must eventually hit sinks. Furthermore, because $G$ has no infinite emitters, the induced downward hereditary graph is finite and every non-sink vertex is singular. This means that any saturated subset that contains the sinks must also contain $v$. Because $v$ was arbitrary, this means that the only saturated subset containing the sinks is all the vertices. Hence, we can apply Theorem~\ref{MoritaFinite}.
\end{proof}

A directed graph $G$ is called \textit{functional} if all vertices have out degree $1$. The next corollary computes the Morita equivalence class of functional graphs.

\begin{corollary}[Functional Graphs]
For $G$ a functional graph, $L_R(G)$ is Morita equivalent to a direct sum $\left(\bigoplus_{i \in I_1} R[x, x^{-1}]\right) \oplus \left(\bigoplus_{i \in I_2} R\right)$ where $I_1$ is the number of connected components with a cycle and $I_2$ is the number of connected components without a cycle.
\end{corollary}

\begin{proof}
If we view $G$ as an undirected graph, we have a set of connected components $C$. It's not hard to prove that these components can be split into two parts: one with exactly one directed cycle $I_1$ and ones with no directed cycles $I_2$.

Again, we know that $L_R(G) \cong \oplus_{X \in C} L_R(X)$ where $X$ is the induced graph by its connected component of vertices. We will use Theorem~\ref{MoritaFinite} to show that when a component $X$ has a directed cycle then $L_R(X) \sim_M R[x, x^{-1}]$ and when $X$ doesn't have a cycle then $L_R (X)\sim_M R$.

Assume that $X$ has a cycle $C$. It's well known that for any cycle $C$, we have that $L_R(C) = M_n(R[x, x^{-1}]) \sim_M R[x, x^{-1}]$. We claim that the cycle $C$ is hereditary and contained in no hereditary saturated subset besides the entirety of $X$. $C$ is clearly hereditary. Let $v \in X$ be arbitrary. Then $v$ must eventually reach the cycle $C$ when following its unique path forward. Because the outdegree was $1$ for all vertices along its unique path, all vertices along the path, including $v$, must be in such a saturated set containing $C$. Thus, we can apply Theorem~\ref{MoritaFinite}.

Now assume that $X$ has no cycle. Let $v \in X$ be arbitrary and consider the path $P$ starting from $v$. This is an infinite path and we know from \cite[Example~1.6.4]{Abrams2017} that $L_R(P) = M_{\infty}(R)$ where $M_{\infty}(R)$ is the set of infinite matrices with finitely many non-zero entries. This is well-known to be Morita equivalent to $R$. Clearly our path $P$ is hereditary. To prove that $P$ is contained in no hereditary saturated subset besides $X$, let $v \in X$ be arbitrary. Because $v$ is in the same component as $P$, there exists some undirected path from $v$ to some vertex of $P$. Because $P$ itself is a path, by similar reasoning as before, we have that the path must be directed from $v$ to $P$. Hence, $v$ must be contained in any saturated subset that contains $P$, so again we can apply Theorem~\ref{MoritaFinite}.
\end{proof}

\section{Desingularization of Labelled Leavitt Path Algebras} \label{AppLabelled}

Labelled spaces were introduced by Bates and Pask in \cite{Bates2005CO} in the context of $C^{\ast}$-algebras. More recently, an algebraic analog has been defined in \cite{Boava2021LeavittPA}. See \cite[Section~5.2]{zhang2025partialactionsgeneralizedboolean} for the definition of a normal weakly left-resolving labelled space, notation used in this section, and a description of the inverse semigroup associated to a labelled space. Our main result of this section will a desingularization result for labelled Leavitt path algebras.

Desingularization is a theorem that all algebras from a combinatorial object are Morita equivalent to an algebra that arises from a ``non-singular'' combinatorial object. Some theorems are easier to prove on these ``non-singular'' objects which makes this equivalence important. Desingularization results for Leavitt path algebras can be found in \cite[Section~5]{abrams2008leavitt} and for Ultragraph algebras in \cite[Section~10]{Firrisa2020MoritaEO}. 

Recently, a desingularization result has been found for labelled spaces in the $C^{\ast}$-algebra case in \cite{banjade2024singularities}. Our construction is similar but not exactly the same. In addition, our result is in a slightly different context because we restrict our objects to be normal weakly left-resolving labelled spaces, whereas \cite{banjade2024singularities} considers weakly left-resolving labelled spaces.

The main theorem is as follows.

\begin{theorem}
Let $(\mathcal E, \mathcal L, \mathcal B)$ be a normal weakly left-resolving labelled space that uses a countable alphabet $\mathcal A$ and such that $\mathcal E^0 \in \mathcal B$. Then there exists a normal weakly left-resolving labelled space $(\mathcal E_F, \mathcal L_F, \mathcal B_F)$ such that $L_R(\mathcal E, \mathcal L, \mathcal B) \sim_M L_R(\mathcal E_F, \mathcal L_F, \mathcal B_F)$ and for all $B \in \mathcal B'$, we have that $B$ is regular in $(\mathcal E_F, \mathcal L_F, \mathcal B_F)$.
\end{theorem}

From now on, we drop normal and weakly-left resolving from our notation, and simply refer to them as labelled spaces. For a labelled space $(\mathcal E, \mathcal L, \mathcal B)$, we first construct an associated labelled space $(\mathcal E_F, \mathcal L_F, \mathcal B_F)$. The assumption that $\mathcal E^0 \in \mathcal B$ will be used to construct our new labelled space. Let $\mathcal A = \{a_1, a_2, \ldots\}$ be countable enumeration of the alphabet. Note that our indexing here starts at $1$.

Define the \textit{true sinks} of $(\mathcal E, \mathcal L, \mathcal B)$ to be \[\mathcal E_{\text{tsink}} = \{v \in \mathcal E^0 \colon \exists B \in \mathcal B \text{ such that } v \in B \text{ and } |\Delta_B| = 0\}\] To start, we define a sequence of sets $X_i \subseteq \mathcal E^0$ for $i = 1, 2, \ldots$ as\[X_i \coloneqq \mathcal E_{\text{tsink}} \cup \{v \in \mathcal E^0 \colon \exists (j \geq i) \text{ such that } r(v, a_j) \neq \emptyset\} \subseteq \mathcal E^0\] Define $X_0 = \mathcal E^0$. Note that $X_{i+1} \subseteq X_i$ for all $i = 0, 1 \ldots$

We define \[\mathcal E_F^0 \coloneqq \bigsqcup_{i=0}^\infty X_i\]

Note that for $v \in \mathcal E^0$, we have that $v$ is in some (possibly infinite) prefix of the $\{X_i\}_{i=0}^{\infty}$. Let $n_v$ be the last $i$ where $v \in X_i$ where we interpret $n_v = \infty$ to be that all $v \in X_i$ for all $i$. Note that $n_v$ is well-defined because $v \in X_0$. For a given vertex $v$ with $n_v$, we denote $v_i$ the copy of $v$ in $X_i$ viewed in $\mathcal E_F^0$. 

For every $v \in \mathcal E^0$ and $i \in [0, n_v)$, we add an edge from $v_i$ to $v_{i+1}$ and label this edge with $b_{i+1}$. Now for every $v \in \mathcal E^0$ and outgoing edge $(v, w) \in \mathcal E^1$ with label $a_i \in \mathcal L$, we add an edge $(v_i, w_0)$ with label $a_i$. Note that $v_i$ is well-defined because $r(v, a_i) \neq \emptyset$ and hence $i \leq n_v$. Our new alphabet is defined as $\mathcal A_F \coloneqq \mathcal A \sqcup \{b_1, b_2, \ldots\}$. 

For a subset $X \subseteq \mathcal E^0$, we define the set $\mathcal B(X) = \{B \cap X \colon B \in \mathcal B\}$. We define $\mathcal B_F$ to be the finite disjoint unions $\bigsqcup_{i=0}^\infty \mathcal B(X_i)$. Elements $B \in \mathcal B_F$ can be represented as $B = \bigsqcup_{i=0}^{\infty} (B_i \cap X_i)$ for some $B_i \in \mathcal B$ where $B_i = \emptyset$ for all but a finite number of $i$. Unions, intersections, and relative complements are all defined pairwise. For $A \in \mathcal B$, we write $A \cap X_i$ to view $A$ in the disjoint image $\mathcal B(X_i) \subseteq \mathcal B_F$.

\begin{lemma}
	Let $B = \bigsqcup_{i=0}^{\infty} (B_i \cap X_i)$ be a representation. Let $i \in [0, \infty)$ and $a_n \in \mathcal A$. The following holds.
	
	\begin{enumerate}
		\item $r(B, b_i) = B_i \cap X_{i+1}$
		\item $r(B, a_n) = r(B_n, a_n) \cap X_0$
	\end{enumerate}
\end{lemma}
\begin{proof} We prove each statement separately.
	\begin{enumerate}		
		\item $r(B, i) = B_i \cap X_i \cap X_{i+1}$. From $X_{i+1} \subseteq X_i$, this is equal to $B_i \cap X_i$.
		
		\item Note that the $X_1$ symbolically denotes being in $\mathcal B(X_1)$ but serves no actual purpose because $X_1 = \mathcal B$. We can calculate that $r(B, a_n) = r(B_n \cap X_n, a_n)$. However, for any $v \in B_n$ with an outgoing edge labelled $a_n$ we have that $v \in X_n$ by definition. Hence, $r(B_n \cap X_n, a_n) = r(B_n, a_n)$ and we are done.
	\end{enumerate}
\end{proof}

The below theorem will use the assumption that $\mathcal E^0 \in \mathcal B$.

\begin{theorem}
	The constructed $(\mathcal E_F, \mathcal L_F, \mathcal B_F)$ is a normal weakly left-resolving labelled space where $(\mathcal B_F)_{\text{reg}} = \mathcal B_F$.
\end{theorem}

\begin{proof}	
	We first show that $\mathcal B_F$ contains $r(a)$ for all $a \in \mathcal A_F$ and that $\mathcal B_F$ is closed under relative ranges. It's easy to see that $r(b_i) = X_{i+1}$ for $i = 0, 1, \ldots$. Because we assumed that $\mathcal E^0 \in \mathcal B$, we have that $X_{i+1} = \mathcal E^0 \cap X_{i+1} \in \mathcal B(X_{i+1})$ which is clearly in $\mathcal B_F$ and hence $r(b_i) \in \mathcal B_F$. Now let $B = \bigsqcup_{i=0}^{\infty} B_i \cap X_i$ and $a_n\in \mathcal L_F$ and $i \in [0, \infty)$. By the lemma $r(B, b_i) = B_i \cap X_{i+1} \in \mathcal B_F$. Furthermore, $r(B, a_n) = r(B_n, a_n) \in \mathcal B = X_0$. 
	
	We now show our space is weakly left resolving. Consider $A = \bigsqcup_{i=0}^{\infty} A_i \cap X_i$ and $B = \bigsqcup_{i=0}^{\infty} B_i \cap X_i$. Note that $A \cap B = \bigsqcup_{i=0}^{\infty} (A_i \cap B_i) \cap X_i$. Hence, we find that for $i \in [0, \infty)$ we have that \[r(A \cap B, b_i) = (A_i \cap B_i) \cap X_{i+1} = (A_i \cap X_{i+1}) \cap (A_i \cap X_{i+1}) =  r(A, b_i) \cap r(B, b_i)\] For $a_n \in \mathcal L$, we calculate \[r(A \cap B, a_n) = r(A_n \cap B_n, a_n) = r(A_n, a_n) \cap r(B_n,  b_n) = r(A, a_n) \cap r(B, b_n)\] where here we use the fact that our original space $(\mathcal E, \mathcal L, \mathcal B)$ was weakly left resolving. It's an easy exercise in labelled spaces to show that the character case implies that this is true for all $\alpha \in \mathcal L_F^{\ast}$. 

	The space is normal because we can define relative complements on any $\mathcal B(X_i)$ using the relative complements from $\mathcal B$.
	
	We now show that all sets in $\mathcal B_F$ are regular. It suffices to show that $0 < |\Delta_B| < \infty$ for all $\emptyset \neq B \in \mathcal \mathcal B_F$. Write $B = \bigsqcup_{i=0}^{\infty} B_i \cap X_i$. 
	
	We first show that $|\Delta_B| > 0$. If there exists some $i > 0$ where $B_i \cap X_i \neq \emptyset$, then either $r(B, a_i) \neq \emptyset$ or $r(B, b_i) \neq \emptyset$, so $\Delta_B$ cannot be empty. If the last statement was not true, then we must have that $B_0  \neq \emptyset$. Assume for the sake of contradiction that $r(B, b_1) = \emptyset$. This means that $B_0 \subseteq \mathcal B_{\text{sink}} \setminus \mathcal B_{\text{tsink}}$. But this means that all elements in $B_0$ are true sinks, which is a contradiction.

	We now show that $|\Delta_B| < \infty$. We note that $r(B, b_i) = B_i \cap X_{i+1}$. But this can only be non-zero for a finite number of $i$ because $B_i = \emptyset$ for all but a finite number of $i$. Furthermore, $r(B, a_n) = r(B_n, a_n) \cap X_0$. Again, because $B_n = \emptyset$ for all but a finite number of $n$, we have that $r(B, a_n) = \emptyset$ for all but a finite number of $a_n$. 
	
	Hence, $|\Delta_B| < \infty$ as well.
	
\end{proof}

We now prove the conditions of Theorem~\ref{finalmorita1}. For ease of notation, let $S_1 \coloneqq  S_{(\mathcal E, \mathcal L, \mathcal B)}$ and $S_2 \coloneqq S_{(\mathcal E_F, \mathcal L_F, \mathcal B_F)}$ as described in \cite[Section~5.2]{zhang2025partialactionsgeneralizedboolean}. We first build the injective homomorphism $S_1 \hookrightarrow S_2$. To do this, we associate to $a_n \in \mathcal A$ the word $[b_1, \ldots, b_n, a_n] \in \mathcal A_F^{\ast}$. This defines a morphism $h: \mathbb F[\mathcal A] \rightarrow \mathbb F[\mathcal A_F]$. Using this $h$, we define a map $S_1 \hookrightarrow S_2$ that takes \[(\alpha, A, \beta) \mapsto (h(\alpha), A \cap X_0, h(\beta))\]

It remains to prove that this is a valid homomorphism. Below are some calculations involving $h$, which we will use without mention.

\begin{lemma} \label{lemma:hcalc} Let $\alpha \in \mathcal L^{\ast}$. Let $B = \bigsqcup_{i=0}^{\infty} B_i \cap X_i$.
	\begin{enumerate}
		\item $r(B, h(\alpha)) = r(B_i, \alpha) \cap X_0$
		\item 	$(\mathcal B_F)_{h(\alpha)} = \mathcal B_{\alpha}(X_0)$
		\item $(\mathcal B_F)_{h(\alpha)b_1\ldots b_i} = \mathcal B_{\alpha}(X_{i})$
	\end{enumerate}
\end{lemma}

\begin{proof}
	These are all easy proofs in the character case for $a \in \mathcal A$ and can be extended to all $\alpha \in \mathcal L^{\ast}$ easily.
\end{proof}

We now present a lemma that will be used to show that the map $S_1 \hookrightarrow S_2$ is a morphism.

\begin{lemma}
	If $\alpha \beta^{-1} \in \mathbb F[\mathcal A]$ for $\alpha, \beta \in \mathcal L^{\ast}$ is a reduced word, then so is $h(\alpha)h(\beta)^{-1} \in \mathbb F[\mathcal A_F]$.
\end{lemma}
\begin{proof}
	If $\alpha$ or $\beta$ is the empty word, this is obvious. Otherwise, assuming that both of them are non-zero, we know that the last element $\alpha_{|\alpha|} = a_{i_\alpha}$ and $\beta_{|\beta|} = a_{i_{\beta}}$ for some values $i_{\alpha}, i_{\beta}$. Note that because $\alpha \beta^{-1}$ is reduced, we have that $\alpha_{|\alpha|} \neq \beta_{|\beta|} \Rightarrow i_{\alpha} \neq i_{\beta}$. Thus, the last term of $h(a_{|\alpha|})$ is $a_{i_{\alpha}}$ which is not equal to the last term of $h(b_{|\beta|}) = b_{i_{\beta}}$. 
\end{proof}

\begin{remark}
	If $\alpha^{-1}\beta$ is reduced, then so is $h(\alpha)^{-1}h(\beta)$ because the inverse operation preserves reduced words.
\end{remark}

\begin{corollary} \label{corollprefix}
	For $\alpha, \beta \in \mathcal L^{\ast}$, If $h(\alpha)$ is a prefix or suffix of $h(\beta)$, then $\alpha$ is a prefix or suffix of $\beta$ respectively.
\end{corollary}
\begin{proof}
	We only prove the prefix version as the suffix version is similar. Let $h(\alpha)$ be a prefix of $h(\beta)$ and consider $h(\alpha)^{-1}h(\beta)$. Note if $h(\alpha)$ is a prefix of $h(\beta)$, this means the reduced word of $h(\alpha)^{-1}h(\beta)$ contains no inverses. We can calculate that $h(\alpha)^{-1}h(\beta) = h(\alpha^{-1}\beta)$. We can reduce $\alpha^{-1}\beta$ to $\alpha'^{-1}\beta'$. Because $\alpha'^{-1}\beta'$ is reduced, we must have that $h(\alpha')^{-1}h(\beta')$ is reduced. Because this is equal to $h(\alpha)^{-1}h(\beta)$ which contains no inverses, we must have that $\alpha' = \omega$, so $\alpha$ must be a prefix of $\beta$.
\end{proof}
\begin{theorem}
	$S_1 \hookrightarrow S_2$ with $(\alpha, A, \beta) \mapsto (h(\alpha), A \cap X_0, h(\beta))$ is a well-defined injective homomorphism of inverse semigroups.
\end{theorem}

\begin{proof}
	To prove that mapping is valid, we need to show that if $A \in \mathcal B_{\alpha} \cap \mathcal B_{\beta}$, then $A \cap X_0 \in (\mathcal B_F)_{h(\alpha)} \cap (\mathcal B_F)_{h(\beta)}$. However, $(\mathcal B_F)_{h(\alpha)} = \mathcal B_{\alpha}(X_0)$ and the same holds for $\beta$. Hence, if $A \in \mathcal B_{\alpha}$ then obviously $A \cap X_0 \in \mathcal B_{\alpha}(X_0)$. The same proof holds for $\beta$, so we are done.
	
	We now prove that $S_1 \rightarrow S_2$ is a homomorphism. Let $a = (\alpha, A, \beta)$ and $b = (\gamma, B, \delta)$ be elements of $S_1$. $h(\alpha, A, \beta) = (h(\alpha), s_0(A), h(\beta))$ and $h(\gamma, B, \delta) = (h(\gamma), s_0(B), h(\delta))$. We calculate each case separately. 
	
	If $\beta = \gamma$, note that $h(\beta) = h(\gamma)$ so the resulting value \[h(a)h(b) = (h(\alpha), (A \cap X_0) \cap (A \cap X_0),h(\delta)) = (h(\alpha), (A \cap B \cap X_0), h(\delta))\] which is obviously the same as $h(ab) = h(\alpha, A \cap B, \delta) = (h(\alpha), (A \cap B) \cap X_0, h(\delta))$. 
	
	If $\gamma = \beta \gamma'$, $h(\gamma) = h(\beta) h(\gamma')$ and then \[h(ab) = (h(\alpha\gamma'), (r(A, \gamma') \cap B) \cap X_0, h(\delta)) = (h(\alpha)h(\gamma'), (r(A \cap X_0, h(\gamma')) \cap (B \cap X_0), h(\delta)))\] We have that \[h(a)h(b) = (h(\alpha), A \cap X_0, h(\beta))(h(\gamma), B \cap X_0, h(\delta))\] One can see that $h(\gamma) = h(\beta)h(\gamma')$, so we apply the same multiplication rule to get 
	
	\[h(a)h(b) = (h(\alpha)h(\gamma'), (r(A \cap X_0, h(\gamma')) \cap B) \cap X_0, h(\delta))\] which is the same. The other case is the exact same, so we omit it.
	
	In the final case where none of the above conditions hold, we can use Corollary~\ref{corollprefix} to show that the resulting product must also be zero, so we are done.
\end{proof}

Hence, we can view $S_1 \subseteq S_2$ as an inverse subsemigroup. The below lemma characterizes which elements of $S_2$ are elements of $S_1$. 

\begin{lemma} \label{lemma:ins1}
	Let $ (\alpha, A, \beta) \in S_2$. Then $(\alpha, A, \beta) \in S_1$ if and only if  $A \in \mathcal B(X_0)$ and, for $\alpha$ and $\beta$, they individually are either empty or their first characters are $b_1$.
\end{lemma}

\begin{proof}
	For $(\alpha', A', \beta') \in S_1$, we have that it's image in $S_2$ is $(h(\alpha'), A' \cap X_0, h(\beta'))$. Clearly then, $A' \in \mathcal B(X_0)$. Furthermore, if $\alpha' = \omega$, then $h(\alpha') = \omega$ and if $\alpha' \neq \omega$, then $h(\alpha')_1 = b_1$. The same proof works for $\beta'$, so we conclude that if $(\alpha, A, \beta) \in S_1$, then the conditions hold.

	In the other direction, assume that $(\alpha, A, \beta) \in S_2$ satisfies our conditions. If $\alpha \neq \omega$, then because $A \in \mathcal B(X_0)$, we must have that $\alpha_{|\alpha|} = a_{i_{\alpha}}$ for some $i_{\alpha}$. A similar fact holds for $\beta$ if $\beta \neq \omega$. With an easy inductive proof, we find that $\alpha = h(\alpha')$ for some $\alpha' \in \mathcal L^{\ast}$ and similarly $\beta = h(\beta')$. Using Lemma~\ref{lemma:hcalc}, we find that $A \in (\mathcal B_F)_{h(\alpha')} \cap (\mathcal B_F)_{h(\beta')}$ implies that $A_0 \in \mathcal B_{\alpha'} \cap \mathcal B_{\beta'}$ and hence $(\alpha', A_0, \beta') \in S_1$ is a valid element and it corresponds to $(\alpha, A, \beta) \in S_2$.

\end{proof}
The following corollary is then obvious.

\begin{corollary} \label{corollary:ine1}
	Let $(\alpha, A, \alpha) \in E_2$. Then $(\alpha, A, \alpha) \in E_1$ if and only if $A \in \mathcal B(X_0)$ and either $\alpha = \omega$ or $\alpha_1 = b_1$.
\end{corollary}

The below lemma gives a sufficient condition for $s \in S_2$ to satisfy $s^+ \cap S_1 \neq \emptyset$. This casework will be used in several proofs.

\begin{lemma} \label{abovecase}
	Let $0 \neq (\alpha, A, \beta) \in S_2$. If one of the following cases holds, then $(\alpha, A, \beta)^+ \cap S_1 \neq \emptyset$.
	
	\begin{enumerate}
		\item $\alpha = \omega$, $\beta = \omega$, and $A \in \mathcal B(X_0)$.
		\item $\alpha = \omega$, $\beta_1 = b_1$, and $A \in \mathcal B(X_0)$ 
		\item $\alpha_1 = \omega$, $\beta_1 = \omega$ and $A \in \mathcal B(X_0)$
		\item $\alpha_1 = b_1$ and $\beta_1 = b_1$.
	\end{enumerate}
\end{lemma}
\begin{proof}
	In the first three cases, by Lemma~\ref{lemma:ins1}, we have that $(\alpha, A, \beta) \in S_1$, so the only interesting case is (4).
	
	We will successively build $s_i$ such $s_0 = (\alpha, A, \beta)$ and each $s_i \leq s_{i+1}$. We will show that the process terminates in some finite number $N$ steps and that $s_N \in S_1$.
	
	To run our process, we first note that if the last character of $\alpha$ is some $b_i$, the last character of $\beta$ must also be $b_i$ because the range of $b_i$ is disjoint with the range of any other character in $\mathcal A_F$. Similarly, if the last character of $\alpha$ is some $a_i$, the last character of $\alpha$ must be some $a_j$ (not necessarily the same character $a_i$, but it can't be some $b_j$). 
	
	Furthermore, note that if $A \in \mathcal B_{\alpha' b_n} \cap \mathcal B_{\beta' b_n}$ then we have that $A \in \mathcal B(X_n)$. By definition there exists $A' \in \mathcal B_{\alpha'} \cap \mathcal B_{\beta'} \subseteq \mathcal B_F$ such that $A \subseteq r(A', b_n)$. Using the representation of $A \in \mathcal B_F$, we can enforce that $A \in \mathcal B(X_n)$. Abusing of notation slightly, we denote $A(X_{n-1})$ to be any choice of $A' \in \mathcal B(X_{n-1})$ where $A\subseteq r(A', b_n)$.
	
	For a word $\alpha = \alpha_1 \ldots \alpha_n \neq \omega$, we define the word $\alpha[:-1] = \alpha_1 \ldots \alpha_{n-1}$ as removing the last characters.
	
	With the preliminaries done, we now describe the process:
	
	\begin{enumerate}
		\item If the last character of $\alpha_i$ is some $b_n$, then let $s_{i+1} = (\alpha_i[:-1], A_i(X_{n-1}), \beta_i[:-1])$
		\item Otherwise, stop
	\end{enumerate}

	Step (1) of the process is well-defined because we've shown that if the last character of $\alpha_i$ is some $b_n$, then so is the last character of $\beta_i$. The process also clearly ends some finite $N$ steps because at every step the finite string $\alpha$ loses a character. Taking the idempotent $e_i = (\alpha_i, A_i, \alpha_i) \in E_2$, we calculate $e_is_{i+1} = (\alpha_i, A \cap r(A_i(X_{n_1}), b_n), \beta_i) = (\alpha_i, A, b_i)$ because $A \subseteq r(A_i(X_{n-1}), b_n)$ by definition. Hence, $s_i \leq s_{i+1}$.
	
	It remains to show that $s_N \in S_1$. Note that $\alpha_N$ and $\beta_N$ are prefixes of $\alpha$ and $\beta$ so they (individually) are either empty or they begin with $b_1$ by assumption. Using Lemma~\ref{lemma:ins1}, it suffices to show that $A \in \mathcal B(X_0)$. By our process, we either have removed no characters or the last removed character must be $b_1$ as after every $b_n$ $(n \geq 2)$ must be a $b_{n-1}$. If we removed no characters, then the last characters of both $\alpha$ and $\beta$ must be some $a_{\alpha}$ and $a_{\beta}$ respectively as we assumed that both $\alpha$ and $\beta$ are not the empty word. Clearly then $A_N \in \mathcal B(X_0)$. If the last removed character was $b_1$, then $A_N = A_{N-1}(X_0)$ which by definition is in $\mathcal B(X_0)$, so we are done.
\end{proof}

\begin{corollary}\label{abovecasecorol}
	Let $0 \neq (\alpha, A, \alpha) \in E_2$. Then $(\alpha, A, \alpha)^+ \cap E_1 \neq \emptyset$ if and only if one of the following holds 
	
	\begin{enumerate}
		\item $\alpha = \omega$ and $A \in \mathcal B(X_0)$
		\item $\alpha_1 = b_1$
	\end{enumerate}
\end{corollary}
\begin{proof}
	The sufficient direction is Lemma~\ref{abovecase}.
	
	To prove that it is necessary, assume that there exists a $(\beta, B, \beta) \in E_1 \subseteq E_2$ such that $(\alpha, A, \alpha) \leq (\beta, B, \beta)$. Then by Lemma~\ref{corollary:ine1}, we have that $B \in \mathcal B(X_0)$ and either $\beta = \omega$ or $\beta_1 = b_1$. We will show that one of the cases must hold. Note that $\beta$ must be a prefix of $\alpha$. 
	
	If $\alpha = \omega$, then $\beta = \omega$ and we must have that $A \subseteq B$ so $A \in \mathcal B(X_1)$. 
	
	If $\alpha \neq \omega$, then we have the following cases. If $\beta \neq \omega$ then $\beta$ begins with $b_1$ so we must have that $\alpha$ begins with $b_1$ as well. If $\beta = \omega$, this means that $A \subseteq r(B, \alpha) \neq \emptyset$. However, because $B \in \mathcal B(X_0)$, we have that $r(B, a) = \emptyset$ for any $a \in \mathcal A_F \setminus \{b_1\}$, so again we must have that $\alpha$ begins with $b_1$.
\end{proof}

We now check the remaining four conditions of Theorem~\ref{finalmorita1}. For ease of checking, we split them into four theorems.

\begin{theorem}
	$E_1 \subseteq E_2$ preserves finite covers. Namely, $S_1 \subseteq_c S_2$.
\end{theorem}
\begin{proof}	
	We use \cite[Lemma~3.18 and Remark~3.19]{zhang2025partialactionsgeneralizedboolean}. Let $(\alpha, A, \alpha) \in E_2$ with $\alpha \in \mathcal L_F^{\ast}$ and $A \in (\mathcal B_F)_{\alpha}$. Let $A = \bigsqcup_{i=0}^{\infty} (A_i \cap X_i)$.
	
	If $(\alpha, A, \alpha)^+ \cap E_1 = \emptyset$ then we are done. By Corollary~\ref{abovecasecorol}, the remaining cases are $\alpha = \omega$ and $A \in \mathcal B(X_1)$ or $\alpha$ begins with $b_1$. 
	
	If $\alpha = \omega$ and $A \in \mathcal B(X_0)$, then $(\alpha, A, \alpha) \in E_1$. 
	
	If $\alpha$ begins with $b_1$ and ends with some $a_i$, then $A \in \mathcal B(X_0)$ so again $(\alpha, A, \alpha) \in E_1$.
	
	Now assume that $\alpha$ ends with $b_n$ for some $n$. Note that $A \in \mathcal B(X_n)$. There are two cases. 
	
	If there is some vertex $v \in A$ with an outgoing edge labelled $a_m$ for $m \geq n$,  we can then calculate that taking $\beta = [b_{n+1}, \ldots, b_m, a_m]$, we have that $r(A, \beta) \neq \emptyset$ and hence $(\alpha\beta, r(A, \beta), \alpha\beta) \in E_1$ is valid because $r(A, \beta) \neq \emptyset$. It's clear that $(\alpha\beta, r(A, \beta), \alpha \beta) \leq (\alpha, A, \alpha)$.
	
	Otherwise, $A \in \mathcal B(X_n)$ contains only true sinks. Let $\alpha'$ be such that $\alpha = \alpha' b_1 \ldots b_n$. Take any $v \in A$ and $B \in \mathcal B$ that contains $v$ such that $|\Delta_B| = 0$ (this exists because $v$ is a true sink). We can then calculate $(\alpha', r(\alpha') \cap B \cap X_0, \alpha')^- = \{(\alpha'b_1\ldots b_m, r(\alpha') \cap B \cap X_m, \alpha'b_1\ldots b_m) \colon m \in [0, \infty)\}$ as $B$ contains only sinks. All have non-zero multiplication with $(\alpha, A, \alpha)$ as $v \in A \cap B$, so we are done.
\end{proof}
\begin{theorem}
$E_1 \subseteq E_2$ is tight.
\end{theorem}
\begin{proof}
	We use \cite[Lemma~3.23]{zhang2025partialactionsgeneralizedboolean}. Let $(\alpha, A, \alpha) \in E(S_2)$. Again by Corollary~\ref{abovecasecorol}, we may assume that $\alpha = \omega$ and $A \in \mathcal B(X_1)$ or $\alpha$ begins with $b_1$. 
	
	Similar to the last proof, if $\alpha = \omega$ and $A \in \mathcal B(X_1)$ or $\alpha_1 = b_1$ and ends with some $a_k$, then $(\alpha, A, \alpha \in )E_1$ and we are done because any element is a cover of itself.
	
	Hence, we may assume that $\alpha_1 = b_1$ and ends with some $b_n$. Let $A' \in \mathcal B$ such that $A = A' \cap X_n$ and let $\alpha'$ be such that $\alpha = \alpha'b_1\ldots b_n$. Note that $\alpha'$ is either empty or the last character of $\alpha'$ is some $a_k$. Let $y = (\alpha', (A' \cap r(\alpha')) \cap X_0,  \alpha')$ where interpret $r(\omega) = \mathcal E^0$. By the previous statements on $\alpha'$ and using Corollary~\ref{corollary:ine1}, we have that $y \in E(S_1)$. 
	
	For $i = 1, \ldots, n-1$, define $\beta_i = [b_1, \ldots, b_i, a_i]$ and consider the set of elements $y_i = (\alpha'\beta_i, r(A' \cap r(\alpha'), \beta_i) \cap X_0, \alpha'\beta_i)$. It is not difficult to see that $y_i \in E(S_1)$ and furthermore we have that $y_i, x \leq y$ and $y_i \wedge x = 0$.
	
	It remains to show that $\{y_1, \ldots, y_{n-1}, x\}$ is a cover for $y$. Let $y' = (\alpha'\gamma, B, \alpha' \gamma) \leq y$ where $\emptyset \neq B \subseteq r(A' \cap r(\alpha'), \gamma)$. We will do casework on $\gamma$. 
	
	If $\gamma$ has a prefix of $\beta_i$, then we are done because $(\alpha'\beta_i, r(A, \beta_i), \alpha'\beta_i)$ is essentially maximal for all elements with prefix $\alpha'\beta_i$. Hence, we can assume that $\gamma$ does not have $\beta_i$ as a prefix. 
	
	The remaining case is when $\gamma$ has a prefix $[1, 2, \ldots, m-1]$ for some $m \geq 1$. We will keep reducing to smaller idempotents until we reach the case $m \geq n$. The case $m \geq n$ is obvious by multiplication. 
	
	If $m < n$, by similar techniques used to prove that all sets are regular, we either have that $r(B, b_m) \neq \emptyset$ or $r(B, a_m) \neq \emptyset$. Hence, we can just either make $(\alpha' a_m, r(B, a_m), \alpha' a_m)$ with prefix $\beta_m$ or $(\alpha'b_m, r(B, b_m), \alpha'b_m)$ with a new $m \coloneqq m+1$.
\end{proof}
\begin{theorem}
	For all $x, y \in E_1$ and $s \in S_2$ there exists $s' \in S_1$ such that $xsy \leq s'$.
\end{theorem}
\begin{proof}
	Let $x = (\alpha_x, A_x, \alpha_x)$ and $y = (\alpha_y, A_y, \alpha_y)$  be in $E_1$. By Corollary~\ref{corollary:ine1}, we have that $A_x, A_y \in \mathcal B(X_0)$ and $\alpha_x, \alpha_y$ are either empty or begin with $b_1$. Let $s = (\alpha, A, \beta) \in S_2$. We will show that $xsy=0$ or $xsy$ satisfies at least one of the conditions of Lemma~\ref{abovecase}.
	
	Assume $0 \neq xsy = (\alpha', A', \beta')$. If $\alpha_x \neq \omega$, then $\alpha'$ must begin with $\alpha_x$ and hence $\alpha'$ begins with $b_1$. If $\alpha'_x = \omega$ then we know that $xs = (\alpha, r(A_x, \alpha) \cap B, \beta)$. If $\alpha \neq $ is not the empty word, then because $A_x \in \mathcal B(X_0)$ and $r(A_x, \alpha) \neq \emptyset$ we find that $\alpha = \alpha'$ begins with $b_1$. A similar proof works to show the same thing for $\beta$ and $\alpha_y$. Hence, if one of $\alpha$ or $\alpha_x$ is not equal to $\omega$, and similarly for $\beta, \alpha_y$, then Lemma~\ref{abovecase} (4) applies.
	
	Hence, the remaining cases are when $\alpha = \alpha_x = \omega$ or $\beta = \alpha_y = \omega$. We will only prove the case $\alpha = \alpha_x = \omega$ as the other case is similar. 
	
	When $\beta = \alpha_y = \omega$, $xsy = (\omega, A_x \cap A \cap A_y, \omega)$ which is Lemma~\ref{abovecase} (1) because $A_x \in \mathcal B(X_0)$. Otherwise, either $\beta$ or $\alpha_y$ is not $\omega$. 
	
	As proven previously, when this is the case we have that $\beta'$ begins with $b_1$. We calculate that $xs = (\omega, A_x \cap B, \beta)$. As $A_x \in \mathcal B(X_0)$, we have that $\beta$ must either be empty or end with some $a_i$. Furthermore, $A'$ is either the intersection of $A_x \cap B$ with something, or the intersection of $A_y$ with something. Hence, in either case, $A' \in \mathcal B(X_0)$.

	There are two remaining cases. When $\alpha_y$ is a prefix of $\beta$, then $\alpha' = \omega$ and we can apply Lemma~\ref{abovecase} (3). If $\beta$ is a strict prefix of $\alpha_y$, then because $\beta'$ ends with $a_i$ for some $i$, we must have that $\alpha'$ begins with $b_1$ so we can apply Lemma~\ref{abovecase} (4).
\end{proof}
\begin{theorem}
	$L_R(S_1) \subseteq L_R(S_2)$ is contained in no proper two-sided ideal
\end{theorem}
\begin{proof}
	Reading the isomorphisms from our inverse semigroups to the labelled Leavitt path algebras with respect to the map $S_1 \rightarrow S_2$ in \cite[Section~5.2]{zhang2025partialactionsgeneralizedboolean}, we find that our subalgebra identifies $p_A \mapsto p_{A \cap X_0}$ for all $A \in \mathcal B$. It's well known that the labelled Leavitt path algebra $L_R(\mathcal E_F, \mathcal L_F, \mathcal B_F)$ is the smallest two-sided ideal containing the projections $p_A$ for $A \in \mathcal B_F$. It's not hard to show that $(s_{b_1}\ldots s_{b_n})^{\ast} p_{A \cap X_0} (s_{b_1}\ldots b_n) = p_{A \cap X_n}$. All elements in $\mathcal B_F$ are finite sums of this, hence all projections are in any two-sided ideal containing the subalgebra the image, and we are done.
\end{proof}

\bibliographystyle{abbrv}
\bibliography{MoritaEquiv}

\begin{thebibliography}{10}

\bibitem{AbramsSurvey}
G.~Abrams.
\newblock Leavitt path algebras: the first decade.
\newblock {\em Bulletin of Mathematical Sciences}, 5, 10 2014.

\bibitem{ABRAMS2007753}
G.~Abrams, G.~{Aranda Pino}, and M.~{Siles Molina}.
\newblock Finite-dimensional leavitt path algebras.
\newblock {\em Journal of Pure and Applied Algebra}, 209(3):753--762, 2007.

\bibitem{abrams2008leavitt}
G.~Abrams and G.~A. Pino.
\newblock The leavitt path algebras of arbitrary graphs.
\newblock {\em Houston J. Math}, 34(2):423--442, 2008.

\bibitem{Abrams01011983}
G.~D. Abrams.
\newblock Morita equivalence for rings with local units.
\newblock {\em Communications in Algebra}, 11(8):801--837, 1983.

\bibitem{banjade2024singularities}
D.~P. Banjade, A.~Chambers, and M.~Ephrem.
\newblock On singularities of labeled graph c*-algebras.
\newblock {\em Oper. Matrices}, pages 205--234, 2024.

\bibitem{Bates2005CO}
T.~Bates and D.~Pask.
\newblock C -algebras of labelled graphs.
\newblock {\em arXiv: Operator Algebras}, 2005.

\bibitem{Boava2021LeavittPA}
G.~Boava, G.~G. de~Castro, D.~Gonçalves, and D.~W. van Wyk.
\newblock Leavitt path algebras of labelled graphs.
\newblock {\em Journal of Algebra}, 2021.

\bibitem{Clark2013EquivalentGH}
L.~O. Clark and A.~Sims.
\newblock Equivalent groupoids have morita equivalent steinberg algebras.
\newblock {\em Journal of Pure and Applied Algebra}, 219:2062--2075, 2013.

\bibitem{CLARK20152062}
L.~O. Clark and A.~Sims.
\newblock Equivalent groupoids have morita equivalent steinberg algebras.
\newblock {\em Journal of Pure and Applied Algebra}, 219(6):2062--2075, 2015.

\bibitem{Firrisa2020MoritaEO}
M.~Firrisa.
\newblock Morita equivalence of graph and ultragraph leavitt path algebras.
\newblock {\em arXiv: Rings and Algebras}, 2020.

\bibitem{Abrams2017}
M.~M. G.~Abrams, P.~Ara.
\newblock {\em Leavitt Path Algebras}.
\newblock Springer London, 2017.

\bibitem{Lawson96}
M.~Lawson.
\newblock Enlargements of regular semigroups.
\newblock {\em Proceedings of The Edinburgh Mathematical Society - PROC
  EDINBURGH MATH SOC}, 39, 10 1996.

\bibitem{Lawson1996EnlargementsOR}
M.~V. Lawson.
\newblock Enlargements of regular semigroups.
\newblock {\em Proceedings of the Edinburgh Mathematical Society}, 39:425 --
  460, 1996.

\bibitem{Morita1958DualityFM}
K.~Morita.
\newblock Duality for modules and its applications to the theory of rings with
  minimum condition.
\newblock 1958.

\bibitem{Muhly87}
P.~Muhly, J.~Renault, and D.~Williams.
\newblock Equivalence and isomorphism for groupoid c* - algebras.
\newblock {\em Journal of Operator Theory}, 17, 01 1987.

\bibitem{Murphy1990-vw}
G.~J. Murphy.
\newblock {\em C*-algebras and operator theory}.
\newblock Academic Press, Sept. 1990.

\bibitem{Renault87}
J.~Renault.
\newblock ReprÉsentation des produits croisÉs d’algÈbres de groupoÏdes.
\newblock {\em Journal of Operator Theory}, 18(1), 1987.

\bibitem{RIEFFEL1974176}
M.~A. Rieffel.
\newblock Induced representations of c*-algebras.
\newblock {\em Advances in Mathematics}, 13(2):176--257, 1974.

\bibitem{rieffel1982morita}
M.~A. Rieffel.
\newblock Morita equivalence for operator algebras.
\newblock In {\em Proc. Symp. Pure Math}, volume~38, pages 285--298, 1982.

\bibitem{SteinbergSemiMorita}
B.~Steinberg.
\newblock Strong morita equivalence of inverse semigroups.
\newblock {\em Houston Journal of Mathematics}, 37, 02 2009.

\bibitem{Steinberg14Modules}
B.~Steinberg.
\newblock Modules over etale groupoid algebras as sheaves.
\newblock {\em Journal of the Australian Mathematical Society}, 97, 05 2014.

\bibitem{Talwar_1995}
S.~Talwar.
\newblock Morita equivalence for semigroups.
\newblock {\em Journal of the Australian Mathematical Society. Series A. Pure
  Mathematics and Statistics}, 59(1):81–111, 1995.

\bibitem{zhang2025partialactionsgeneralizedboolean}
A.~Zhang.
\newblock Partial actions on generalized boolean algebras with applications to
  inverse semigroups and combinatorial $r$-algebras.
\newblock 2025.

\bibitem{AnhMoritaEquivalenceLocal}
P.~N. Ánh and L.~Márki.
\newblock Morita equivalence for rings without identity.
\newblock {\em Tsukuba Journal of Mathematics}, 1987.

\end{thebibliography}

\end{document}